\documentclass[smallextended,envcountsect]{svjour3}

\smartqed 
\usepackage{graphicx}
\usepackage{amsmath}
\usepackage{amsfonts,verbatim,mathdots}
\usepackage{amssymb}
\usepackage{dsfont}

\usepackage{array,multirow,makecell}

\usepackage{mathrsfs}
\usepackage{stmaryrd}

\usepackage{cite}

\usepackage{rotating}
\usepackage{ifthen}
\usepackage{booktabs}
\usepackage{xcolor}
\usepackage{subcaption}
\RequirePackage{caption, float}
\usepackage{graphicx}
\graphicspath{{./}{figs/}}
\usepackage{float}
\newcommand{\f}{\ensuremath{f}}
\newcommand{\n}{\ensuremath{n}}

\newcommand{\ie}{i.e.} 

\newcommand{\markupdraft}[2]{
    \ifthenelse{\equal{#1}{display}}{#2}{}
    \ifthenelse{\equal{#1}{color}}{\color{#2}}{}
}

\newcommand{\newcolored}[3][]{{\markupdraft{color}{#2}#3}
    \ifthenelse{\equal{#1}{}}{}{\markupdraft{display}{{\color{yellow!70!black}[#1]}}}} 

\newcommand{\del}[2][]{{\markupdraft{display}{{\color{yellow!80!black}{\small #2}}}}} 
\newcommand{\nnew}[2][]{\newcolored[#1]{red}{#2}}

\newcommand{\indraftonly}[1]{{#1}}  

\renewcommand{\indraftonly}[1]{}\renewcommand{\markupdraft}[2]{}  
\renewcommand{\del}[2][]{}

\newcommand{\done}[1]{}

\newcommand{\cheikh}[1]{\indraftonly{\color{orange!60!black}~#1$_{-\!\mathrm{Cheikh}}$}}

\newcommand{\ba}{\begin{eqnarray}}
\newcommand{\ea}{\end{eqnarray}}
\newcommand{\baStar}{\begin{eqnarray*}}
\newcommand{\eaStar}{\end{eqnarray*}}

\newcommand{\sublevel}{\mathcal{L}^{\leq}}
\newcommand{\level}{\mathcal{L}}

     \newcommand{\foncfast}[4]{#2 \in #1 \mapsto #3 \in #4}
  \newcommand{\foncsmall}[2]{#1 \mapsto #2}

   \newcommand{\fun}[4]{\begin{array}{l r c l}
& #1 & \longrightarrow & #2 \\
    & #3 & \mapsto & #4 \end{array}}
    
    \newcommand{\sbar}{\rule[-1pt]{0.4pt}{9pt}}   
\newcommand{\dsbar}{\sbar\,\sbar}       
    \newcommand{\norm}[1]{\,\dsbar\,#1\,\dsbar\,}

\newcommand{\dsp}{\displaystyle}
\newcommand{\acco}[1]{\left\{#1\right\}} 
\newcommand{\pare}[1]{\left(#1\right)}
\newcommand{\croc}[1]{\left[#1\right]}
\newcommand{\abs}[1]{\left\lvert#1\right\rvert}

\def\R{{\mathbb R}}
\def\diff{{\mathrm d}}

\def\Z{{\mathbb Z}}
\def\Q{{\mathbb Q}}
\def\B { {\mathcal B } }
\def\S { {\mathcal S } }
\def\image {\textrm{Im} }

\usepackage{authblk}

\journalname{JOTA}

\newcommand{\ph}{PH}
\newcommand{\si}{SI}
\newcommand{\phone}{\ph$_1$}
\newcommand{\phalpha}{\ph$_\alpha$}

\begin{document}

\title{
Scaling-invariant Functions versus Positively Homogeneous Functions
}


\author{Cheikh Toure \and Armand Gissler \and  Anne Auger \and Nikolaus Hansen}
\institute{
Inria and CMAP, Ecole Polytechnique, IP Paris, France \\
firstname.lastname@inria.fr\\cheikh.toure@polytechnique.edu
}

\date{Received: date / Accepted: date}

\maketitle

\begin{abstract}
Scaling-invariant functions preserve the order of points when the points are scaled  by the same positive scalar (usually with respect to a unique reference point).
Composites of strictly monotonic functions with positively homogeneous functions are scaling-invariant with respect to zero. We prove in this paper that also the reverse is true for large classes of scaling-invariant functions. Specifically, we give necessary and sufficient conditions for scaling-invariant functions to be composites of a strictly monotonic function with a positively homogeneous function. We also study sublevel sets of scaling-invariant functions generalizing well-known properties of positively homogeneous functions.
\end{abstract}

~\\Communicated by Juan-Enrique Martinez Legaz.

\keywords{scaling-invariant function \and positively homogeneous function \and compact level set.}
\subclass{49J52 \and  54C35}


\InputIfFileExists{toggles.tex}{}{}

\section{Introduction}
A function $\f:\mathbb{R}^n\to\mathbb{R}$ is scaling-invariant (\si) with respect to a reference point $x^{\star}\in\mathbb{R}^n$ if for all $x,y\in\mathbb{R}^n$ and $\rho>0$:
\begin{align}
\f(x^{\star} + x)\leq f(x^{\star} + y) \iff f(x^{\star} + \rho x )\leq f(x^{\star} + \rho y)
\enspace,
\label{scaling-invariance-definition}
\end{align}
that is, the $f$-order of any two points is invariant under a multiplicative change of their distance to the reference point---the order only depends on their direction and their \emph{relative} distance to the reference.
Scaling-invariant functions appear naturally when studying the convergence of comparison-based optimization algorithms where the update of the state of the algorithm is using $\f$ only through comparisons of candidate solutions~\cite{auger2016linear,fournier2011lower}. A famous example of a comparison-based optimization algorithm is the Nelder-Mead method~\cite{nelder1965simplex}.

A function $p:\mathbb{R}^n\to\mathbb{R}$ is positively homogeneous (\ph) with degree $\alpha > 0$ (\phalpha) if for all $x\in\mathbb{R}^n$ and $\rho>0$:
\begin{equation}\label{eq:ph}
p(\rho \,x)=\rho^\alpha p(x)
\enspace.
\end{equation}
Positively homogeneous functions are scaling-invariant with respect to $x^\star = 0$. We also consider that $x \mapsto p(x-x^\star)$ is positively homogeneous w.r.t.\ $x^\star$ when $p$ is positively homogeneous.
Linear functions, norms, and convex quadratic functions are positively homogeneous.
We can define \ph\ functions piecewise on cones or half-lines, because a function is \phalpha\ if and only if \eqref{eq:ph} is satisfied within each cone or half-line (which is not the case with \si\ functions where $x$ and $y$ in \eqref{scaling-invariance-definition} can belong to different cones).
For example, the function $p: \R^{\n} \to \R$ defined as $p(x) = x_{1}$ if $x_{1} x_{2} > 0$ and $p(x) = 0$ otherwise, is \phone.
Positively homogeneous functions and in particular increasing
positively homogenous functions are well-studied in the context of Monotonic Analysis~\cite{dutta2004monotonic,rubinov1998duality,rubinov2003strictly} or nonsmooth analysis and nonsmooth optimization~\cite{gorokhovik2016positively}.
Specifically, non-linear programming problems where the objective function and constraints are positively homogeneous are analyzed in~\cite{lasserre2002mathematical} whereas saddle representations of continuous positively homogeneous functions by linear functions are established in~\cite{gorokhovik2018saddle}.
The (left) composition of a \ph\ function with a strictly monotonic function is \si\ while this composite function is in general not \ph. One of the questions we investigate in this paper is to which extent \si\ functions and
composites of \ph\ functions with strictly monotonic functions are the same.
We prove that a \emph{continuous} \si\ function is always the composite of a strictly monotonic function with a \ph\ function.
We give necessary and sufficient conditions for an \si\ function to be the composite of a strictly monotonic function with a \ph\ function in the general case.

Only level sets or sublevel sets matter to determine the difficulty of an \si\ problem optimized with a comparison-based algorithm. We investigate different properties of level sets thereby generalizing properties that are known for \ph\ functions, including a formulation of the Euler homogenous function theorem that holds for \ph\ functions.

\subsection*{Notation\del{s:}}
We denote $\R_{+}$ the interval $[0,+\infty)$, $\R_{-} = (-\infty, 0 ]$, $\Z$ the set of all integers, $\mathbb{Z}_{+}$ the set of all non-negative integers and $\Q$ the set of rational numbers. The Euclidean norm is denoted by $\|.\|$. For $x\in\R^{\n}$ and $\rho > 0$, we denote by $\B\pare{x, \rho} = \acco{ y\in \R^{\n}; \| x-y \| < \rho }$ the open ball centered at $x$ and of radius $\rho$,
$\overline{\B\pare{x, \rho}}$ its closure and $\S\pare{x, \rho}$ its boundary. When they are centered at $0$, we denote $\B_{\rho}=\B\pare{0, \rho}$, $\overline{\B}_{\rho}=\overline{\B\pare{0, \rho}}$ and $\S_{\rho} = \S\pare{0, \rho}$.
We refer to a proper interval containing more than a single element as \emph{nontrivial} interval.
For a nontrivial interval $I \subset \R$ and a function $\varphi: I \to \R$, we use the terminology of strictly increasing (respectively strictly decreasing) if for all $a, b \in I$ with $a < b$, $\varphi(a) < \varphi(b)$ (respectively $\varphi(a) > \varphi(b)$). For a real number $\rho$ and a subset $A \subset \R^{\n}$, we define $\rho A = \acco{\rho\, x ; \,x \in A}$. For a function $\f$, we denote by \textrm{Im}($\f$) the image of $\f$.

\section{Preliminaries}

Given a function $\f:\R^{\n} \to \R$ and $x\in \R^{\n}$, we denote the level set going through $x$ as $ \mathcal{L}_{\f,x} = \acco{ y\in \R^{\n}, \f(y) = \f(x)}$ 
 and the sublevel set as
$ \sublevel_{\f,x} = \acco{ y\in \R^{\n}, \f(y) \leq \f(x) }$. 

If $\f$ is \si\ with respect to $x^\star$, then the function $x \mapsto \f(x+x^{\star}) - \f(x^{\star})$ is scaling invariant with respect to $0$. Hence, if a function \f\ is \si, we assume in the following that $\f$ is \si\ with respect to the reference point $0$ and that $f(0)=0$, without loss of generality.

We can immediately imply from \eqref{scaling-invariance-definition} that if $x$ and $y$ belong to the same level set, then $\rho x$ and $\rho y$ belong to the same level set.  Hence the level set of $x$ and $\rho x$ are scaled from one another, \ie, $\mathcal{L}_{\f,\rho x} = \rho \mathcal{L}_{\f,x} $.

Similarly, since for any $x, y\in \R^{\n}$ and $\rho > 0$, $\f( \frac{y}{\rho} ) \leq \f(x)$ if and only if $ \f(y ) \leq \f(\rho x)$, 
\begin{align}
\sublevel_{f,\rho x} = \rho \sublevel_{f,x} \text{\del{ ,} and }  \mathcal{L}_{\f, \rho x} = \rho \mathcal{L}_{\f, x}\enspace.
\label{homothetic}
\end{align}
These properties are visualized in Figure~\ref{fig:scalingInvariance}.

\begin{figure*}
\newcommand{\figwidth}{0.24}
\centering
        \includegraphics[height=\figwidth\textwidth]{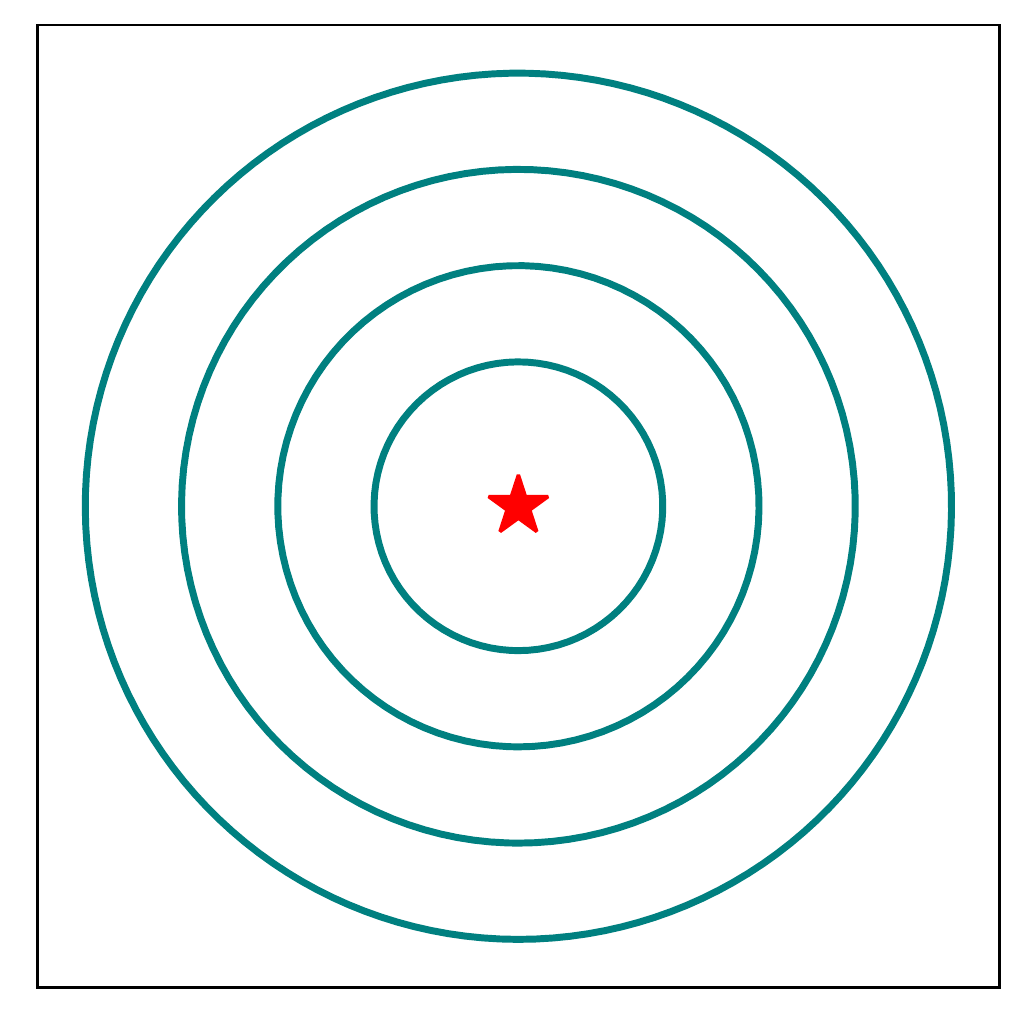}
        \includegraphics[height=\figwidth\textwidth]{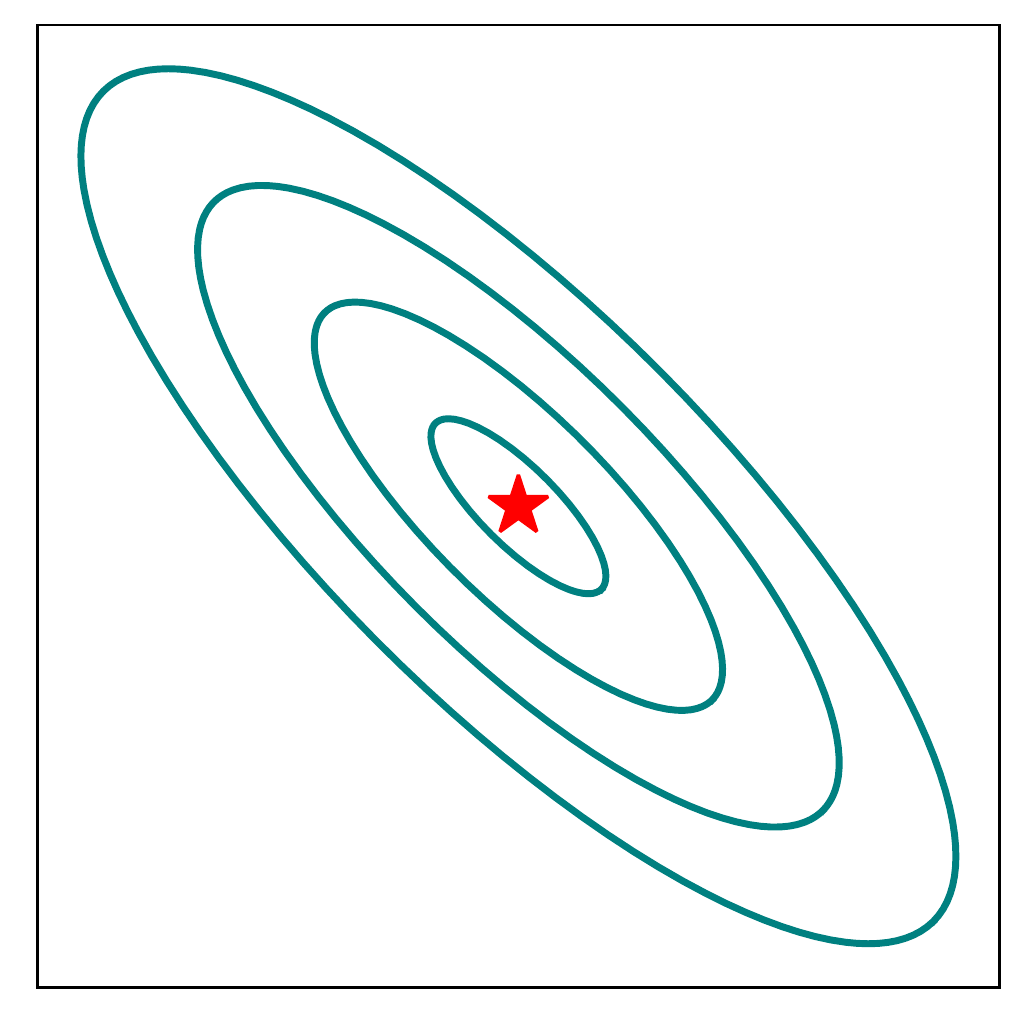}
        \includegraphics[height=\figwidth\textwidth]{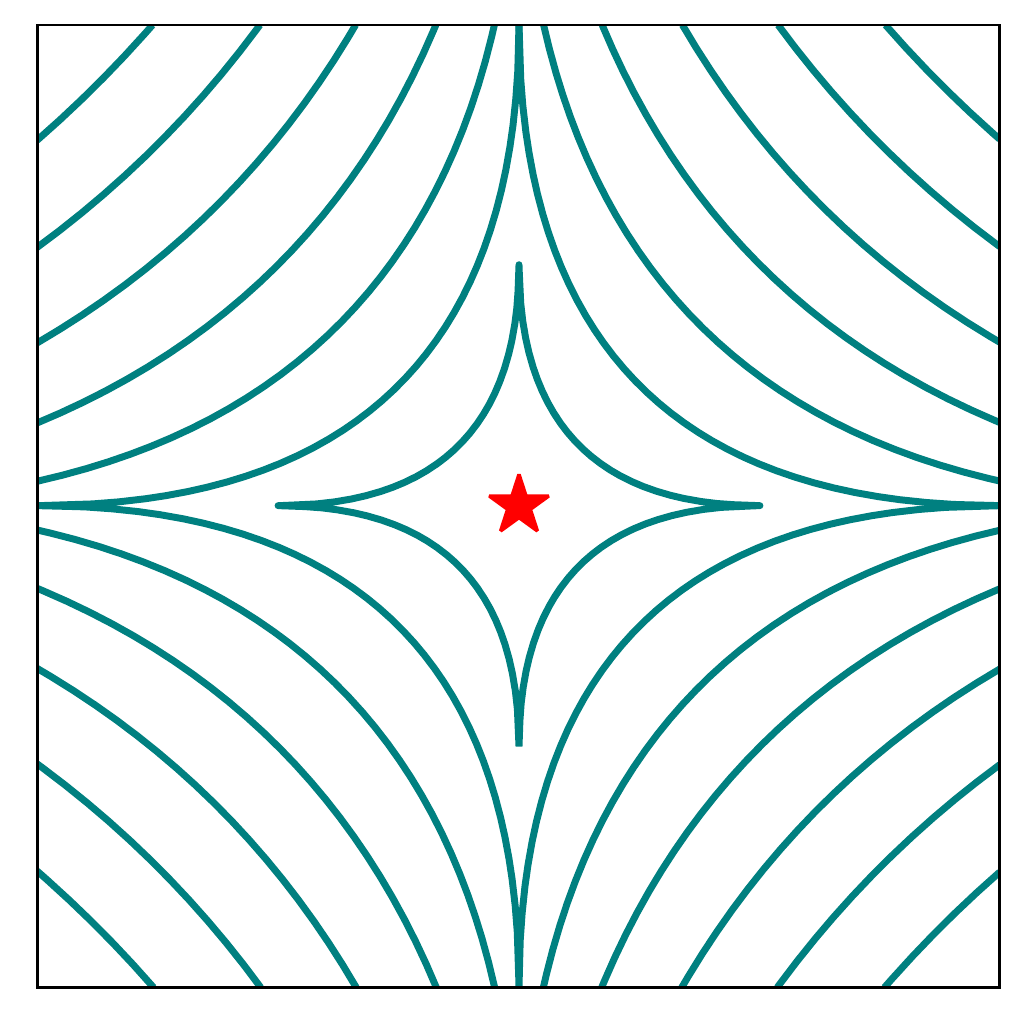}
        \includegraphics[height=\figwidth\textwidth]{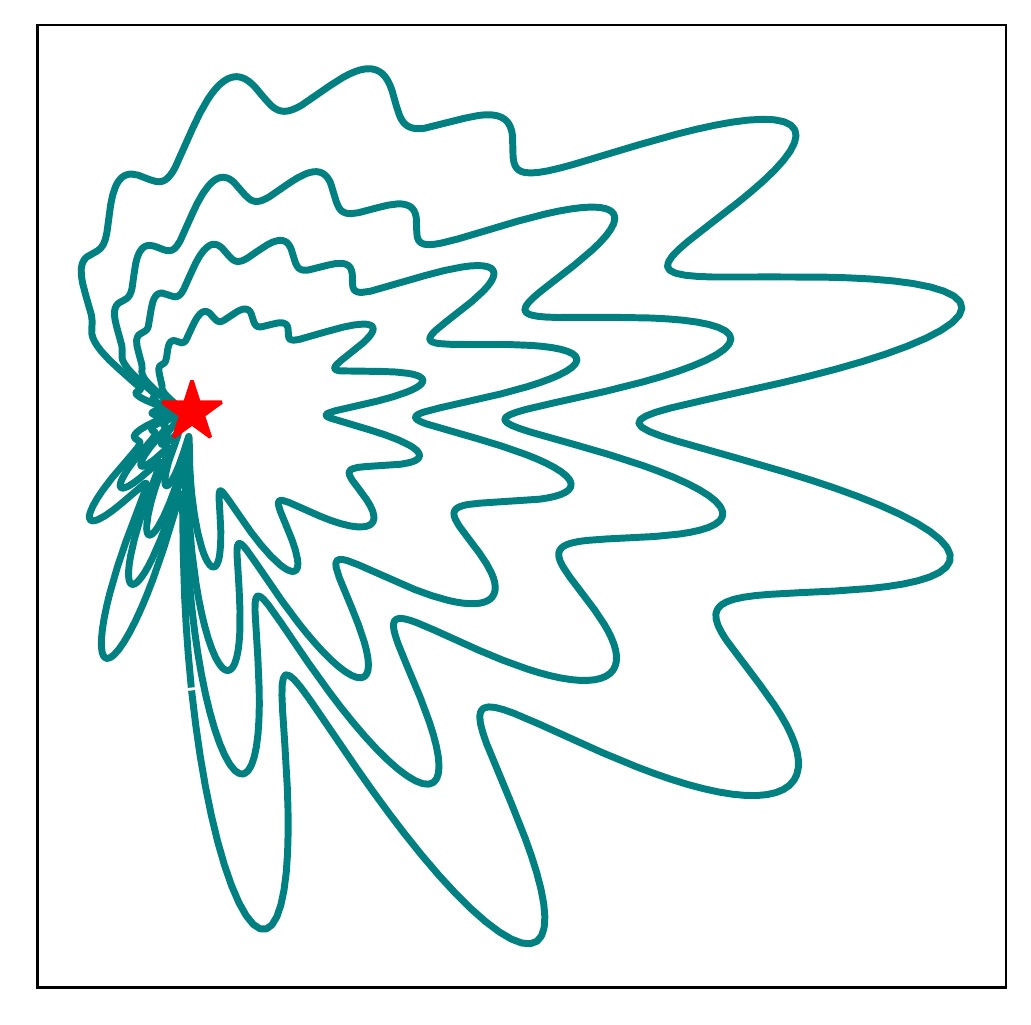}
        \caption{\label{fig:scalingInvariance}
         Level sets of \si\ functions with respect to the red star $x^{\star}$. The four functions are strictly increasing transformations of $x \mapsto p(x - x^{\star})$ where $p$ is a \ph\ function. From left to right: $p(x) = \|x\| ;$ $p(x) = x^\top A x$
for $A$ symmetric positive and definite; $p(x) = \pare{\sum_{i} \sqrt{\abs{x_{i}}} }^{2}$ the $\frac12$-norm; a randomly generated \si\ function from a ``smoothly'' randomly perturbed sphere function. The two first functions from the left have convex sublevel sets, contrary to the last two.
    }
\end{figure*}

Given an \si\ function $\f$, we define surjective restrictions of $\f$ to half-lines along a vector $x \in \R^\n$ as
\begin{equation}\label{eq:def-fx}
\f_x:  t \in [0,\infty)\mapsto f(tx)
\enspace.
\end{equation}
It is immediate to see that the $\f_x$ are also \si\footnote{
This directly follows because for $s, t \in \R_{+}$ and $\rho > 0$,
$ \f_{x}(t) \leq \f_{x}(s) \iff \f(tx) \leq \f(sx) \iff \f(\rho t x) \leq \f(\rho s x) \iff \f_{x}(\rho t) \leq \f_{x}(\rho s)
$.}.
 However, \f\ may not be \si\ even when all $\f_x$ are\footnote{
 For example, define $f: \R\to\R$ as $t \mapsto t$ on $\R_{+}$ and $t \mapsto t^{2}$ on $\R_{-}$. Then $\f_1(t) = t$ and $\f_{-1}(t) = t^2$, for $t\in\R_+$, are both \si\ and even \ph\ with degree 1 and 2, respectively.
But \f\ is not \si, and hence also not \ph, because $\f(\frac12) = \frac 12  > \frac 14 = \f(-\frac12)$ but $\f(4 \times \frac12) = 2 < 4 = \f(4 \times (-\frac12))$.}.

\newcommand{\strict}{isolated}
\newcommand{\astrict}{an isolated}
Scaling invariant functions have at most one \strict\ local optimum \cite{auger2016linear} where \astrict\ local optimum, say, an isolated argmin, $x$, for a function $g: \R^{\n} \to \R$ is defined in that there exists $\epsilon > 0$ such that for all $y \in \B\pare{x, \epsilon}\setminus\acco{x}$, $g(y) > g(x)$. This result is reminded in the following proposition.

\begin{proposition}[see {\cite[Proposition 3.2]{auger2016linear}}]
Let $f$ be an \si\ function. Then $f$ can admit \astrict\ local optimum only in $f(0)=0$ and this local optimum is also the global optimum. In addition, the functions $f_x$ cannot admit a local plateau, \ie, a ball where the function is locally constant, unless the function is equal to $0$ everywhere.
\label{notplateau}
\end{proposition}

We characterize in the following the functions $\f_x$ of an \si\ function \f\ under different conditions.
\begin{proposition}
If $f$ is a continuous \si\ function on $\R^{\n}$, then for all $x \in \R^n$, $f_x$ is either constant equal to $0$ or strictly monotonic.

More specifically, if $\varphi :\R_+ \to \R$ is a 1-dimensional continuous \si\ function, then $\varphi$ is either constant equal to $0$ or strictly monotonic.

 \label{continuity-monotone}
 \end{proposition}
 \begin{proof}
Assume that $\varphi$ is not strictly monotonic on $[0,\infty)$. Then $\varphi$ is not strictly monotonic on $(0, \infty)$. Since continuous injective functions are strictly monotonic, $\varphi$ is not injective on $(0, \infty)$. Therefore there exists $0 < s < t$ such that $\varphi(s) = \varphi(t)$. By scaling-invariance, it follows that $\varphi(\frac{s}{t}) = \varphi(1)$. It follows iteratively that for all integer $k > 0$, $\varphi\pare{ (\frac{s}{t})^{k} } = \varphi(1)$. Taking the limit for $k \to \infty$, we obtain that $\varphi(0) = \varphi(1)$. Thereby by scaling-invariance again, it follows for all $\rho > 0$ that $\varphi(0) = \varphi(\rho)$. Hence we have shown that if $\varphi$ is not strictly monotonic, it is a constant function.

Now if $\f$ is a continuous \si\ function on $\R^{\n}$ and $x \in \R^{\n}$, then $\f_x$ is also scaling invariant and continuous. Then it follows that $\f_{x}$ is either constant or strictly monotonic.
\qed
\end{proof}

We deduce from Proposition~\ref{continuity-monotone} the next corollary.
\begin{corollary}
Let $\f$ be a continuous \si\ function. If $\f$ has a local optimum at $x$, then for all $t \geq 0$, $\f(tx) = \f(0)$.
In particular, if $f$ has a global argmin (resp. argmax), then $0$ is a global argmin (resp. argmax).
\label{localGlobal}
\end{corollary}
\begin{proof}
Assume that there exists a local optimum at $x$. Then $\f_{x}$ has a local optimum at $1$. Therefore $\f_{x}$ is not strictly monotonic, and thanks to Proposition~\ref{continuity-monotone}, $\f_{x}$ is necessarily a constant function. In other words, $\f(tx) = \f(0)$  for all $t \geq 0$. 
\qed \end{proof}

We derive another proposition with the same conclusions as Proposition~\ref{continuity-monotone} but under a different assumption.
We start by showing the following lemma.
\begin{lemma}
Let $\varphi :\R_+ \to \R$ be an \si\ function continuous at $0$ and strictly monotonic on a nontrivial interval $I \subset \R_{+}$, then $\varphi$ is strictly monotonic.
\label{lemma-monotone}
\end{lemma}
\begin{proof}
Assume without loss of generality that $\varphi$ is strictly increasing on $I$ and that $I = (a, b)$ with $0 < a < b$, up to replacing $I$ with a subset of $I$. Denote $\rho = \frac{b}{a}$. Then $\acco{\croc{\rho^{k}, \rho^{k+1}} }_{k \in \Z}$ covers $\pare{0, \infty}$. To prove that $\varphi$ is strictly increasing on $\pare{0, \infty}$, it is enough to prove that $\varphi$ is strictly increasing on $[ \rho^{k}, \rho^{k+1}]$ for all integer $k$.

Let $k$ be an integer and $(x, y)$ two real numbers such that $\rho^{k} \leq x < y \leq \rho^{k+1}$. Then $a \leq \frac{ax}{\rho^{k}} < \frac{ay}{\rho^{k}} \leq a \rho = b$. Therefore $\varphi(\frac{ax}{\rho^{k}}) < \varphi(\frac{ay}{\rho^{k}})$. And by scaling-invariance, $\varphi(x) < \varphi(y)$.

With the continuity at $0$, it follows that $\varphi$ is strictly increasing on $\R_{+}$. \qed
\end{proof}

We derive from Lemma~\ref{lemma-monotone} the following proposition.

\begin{proposition}
Let $f$ be an \si\ function continuous at $0$. Assume that each $\f_x$ is on some nontrivial interval either strictly monotonic or constant\del{ on some non-empty interval}. Then for all $x \in \R^{\n}$, $f_x$ is either constant equal to $0$ or strictly monotonic.
\label{monotone-monotone}
\end{proposition}

Note that the continuity of the function $\f_{x}$ alone does not suffice to conclude that $f_x$ is either constant or strictly monotonic on some nontrivial\del{-empty} interval. Indeed there exist 1-D continuous functions (even differentiable functions) that are not monotonic on any nontrivial interval~\cite{hardy1916weierstrass, denjoy1915fonctions, buskes1997topological}.

For the sake of completeness, we construct \si\ functions in $\R_{+}$ that are not monotonic on any nontrivial interval.
The construction of such functions is based on the nonlinear solutions of the Cauchy functional equation: for all $x, y \in \R$,
$
g(x+y) = g(x) + g(y)
$,
called Hamel functions~\cite{kuczma2009introduction}. A Hamel function $f$ also satisfies
$\f(q_{1}x + q_{2}y ) = q_{1}\f(x) + q_{2}\f(y)$
for all real numbers $x,y$ and rational numbers $q_{1}, q_{2}$~\cite[Chapter 2]{aczel1966lectures}. Since $g$ is nonlinear, there exist real numbers $x$ and $y$ such that the vectors $\{(x,g(x)), (y,g(y))\}$ form a basis of $\R^{2}$ over the field $\R$. Then the graph of $g$, which is a vector subspace of $\R^{2}$ over the field $\Q$, contains $\acco{ q_{1}\cdot(x,g(x)) + q_{2} \cdot (y,g(y)) ; (q_{1},q_{2}) \in \Q^{2}}$ which is dense in $\R^{2}$. Therefore a 1-D Hamel function is highly pathological, since its graph is dense in $\R^2$.

\begin{lemma}
There exist \si\ functions on $\R_{+}$ that are neither monotonic nor continuous on any nontrivial interval.
\label{lemma-uncountable}
\end{lemma}
\begin{proof}
We start by choosing a nonlinear solution of the Cauchy's functional equation denoted by $g : \R \to \R$, knowing that there are uncountably many ways to pick such a $g$~\cite{kuczma2009introduction}.
Then for all real numbers $a$ and $b$, $g(a+b) = g(a) + g(b)$. And since $g$ is not linear, we also know that g is neither continuous nor monotonic on any nontrivial interval~\cite{kuczma2009introduction}.
Let us define $\f = \exp\circ\, g \circ \log $ on $(0, \infty)$ and $\f(0) = 0$. Then $\f(x) > 0$ for all $x > 0$ and $\f$ is still not monotonic on any nontrivial interval.
We also have for all $\rho > 0$ and $x > 0$, 
$\f(\rho x) = \exp\pare{ g (\log(x) + \log(\rho)) } = \exp(g(\log(x))) \exp(g(\log(\rho))) = f(x) f(\rho)$.
This last result gives the scaling-invariance property.
\qed 
\end{proof}

Based on Lemma~\ref{lemma-uncountable}, we derive the next proposition.

\begin{proposition}
There exist $\si$ functions $\f$ on $\R^{\n}$ such that for all non-zero $x$, $\f_{x}$ is neither monotonic nor continuous on any nontrivial interval.
\end{proposition}
\begin{proof}
Based on Lemma~\ref{lemma-uncountable}, there exists $\varphi: \R_{+} \to \R$ \si\ on $\R_{+}$ which is neither monotonic nor continuous on any nontrivial interval. We construct $\f$ as follows. For all $x \in \R^{\n}$, $\f(x) = \varphi(\|x\|)$. Then $\f$ is \si\ because for $x, y \in \R^{\n}$ and for $\rho > 0$,
$\f(x) \leq \f(y) \iff \varphi(\|x\|) \leq \varphi(\|y\|) \iff \varphi(\rho\|x\|) \leq \varphi(\rho\|y\|) \iff \f(\rho x) \leq \f(\rho y)$. In addition for a non-zero $x$ and $t \geq 0$, $\f_{x}(t) = \f(tx) = \varphi(t\|x\|)$ and then $\f_{x}$ is neither monotonic nor continuous on any nontrivial interval. \qed
\end{proof}

Now assume that $\f$ is a continuous scaling invariant function and we can write $\f = \varphi \circ g$ where $\varphi$ is a continuous bijection and $g$ is a positively homogeneous function. As a direct consequence of the bijection theorem given in Appendix~\ref{sec:appendixA},\del{Theorem~\ref{homeomorphismTheorem}} $\varphi^{-1}$ is continuous if $\varphi$ is a continuous bijection defined on an interval. Therefore $g = \varphi^{-1}\circ \f$ is also continuous. This result is stated in the following corollary.
\begin{corollary}
Let $\f$ be a continuous \si\ function, $\varphi$ a continuous bijection defined on an interval in $\R$ and $p$ a positively homogeneous function such that $\f = \varphi \circ p$. Then $p$ is also continuous.
\label{continuous-positively}
\end{corollary}


\section{Scaling-invariant Functions as Composite of Strictly Monotonic Functions with Positively Homogeneous Functions}
\label{transformations}

As underlined in the introduction, compositions of strictly monotonic functions with positively homogeneous functions are scaling-invariant (\si) functions. We investigate in this section under which conditions the converse is true, that is, when \si\ functions are compositions of strictly monotonic functions with \ph\ functions.
Section~\ref{continuous-transformations} shows that continuity is a sufficient condition, whereas Section~\ref{sec:iff} gives some necessary \emph{and} sufficient condition on \f\ to be decomposable in this way.

\subsection{Continuous \si\ Functions}
\label{continuous-transformations}

We prove in this section a main result of the paper: any continuous \si\ function $\f$ can be written as $\f = \varphi \circ p$ where $p$ is \phone\ and $\varphi$ is a homeomorphism (and in particular strictly monotonically increasing and continuous).
The proof relies on the following proposition where we do not assume yet that $\f$ is continuous but only the restrictions of $\f$ to the half-lines originating in $0$, the $\f_x$ functions. 
\begin{proposition}
Let $\f$ be an {\rm \si} function such that for any $x\in\mathbb{R}^n$, $f_x$ as defined in \eqref{eq:def-fx} is continuous and strictly monotonic or constant\del{either constant or strictly monotonic and continuous}. Then for all $\alpha > 0$, there exist a {\rm \phalpha} function $p$ and a strictly increasing, continuous bijection (thus a homeomorphism) $\varphi$ such that
$\f = \varphi \circ p$.
For a non-zero $\f$ and $\alpha > 0$, the choice of $(\varphi, p)$ is unique up to a left composition of $p$ with a piece-wise linear function.

\begin{itemize}
\item[{\rm (i)}] In addition, if all non-constant $\f_{x}$ have the same monotonicity for all $x \in \R^{\n}$, then for any $x_{0}\in \R^{\n}$ such that $\f(x_{0}) \neq 0$, the homeomorphism $\varphi$ corresponding to a {\rm \phone} function can be chosen as
$\f_{x_{0}}$ and is at least as smooth as $\f$.

\item[{\rm (ii)}] Otherwise, there exist $x_{1}, x_{-1} \in \R^{\n}$ such that $\f_{x_{1}}$ is strictly increasing and $\f_{x_{-1}}$ is strictly decreasing. And for any such $x_{1}$ and $x_{-1}$ we can choose as homeomorphism $\varphi$ corresponding to a {\rm \phone} function the\del{ following} function\del{ equal to} $\f_{x_{1}}$ on $\R_{+}$ and\del{ equal to} $t \mapsto \f_{x_{-1}}(-t)$ on $\R_{-}$.
\end{itemize}

\label{DarbouxGeneral}
\end{proposition}

\begin{proof}
Let $\f$ be an \si\ function such that for any $x\in\mathbb{R}^n$, $\f_x$ is either a constant or a strictly monotonic continuous function. 

In the case where all the $\f_{x}$ are constant for all $x \in \R^{\n}$, then $\f=0$ and therefore we can take $p_{\alpha} = 0$ as a candidate for a continuous \phalpha\ and $\varphi_{\alpha}: t \mapsto t$ as the candidate for the corresponding homeomorphism.

From now on, at least one of the $\acco{\f_{x}}_{x \in \R^{\n}}$ is non-constant.
We now split the proof in two parts, the case where all the non-constant $\f_{x}$ have the same monotonicity and the case where there exist $x_{1}, x_{-1} \in \R^{\n}$ such that $\f_{x_{1}}$ is strictly increasing and $\f_{x_{-1}}$ is strictly decreasing.

{\it Part 1.} Assume here that all the non-constant $\f_{x}$ have the same monotonicity for all $x \in \R^{\n}$.
And up to a transformation $x \mapsto -\f(x)$, we can assume without loss of generality that they are increasing. Therefore $0$ is a global argmin and since we have assumed $f(0)=0:$ $\f(x) \geq 0$ for all $x \in \R^{\n}$. Then there exists $x_{0} \in \R^{\n}$ such that $\f(x_{0}) > 0$.

For any $x\in\level_{\f,x_{0}} = \acco{y\in\R^{\n}, \f(y) = \f(x_{0})}$, and any $\lambda >0$ different than 1, $\lambda x\notin\level_{\f,x_{0}}$. Indeed, as $x\in\level_{\f,x_{0}}$,  we know from Proposition~\ref{continuity-monotone} that $\f_{x}$ is strictly increasing on $\mathbb{R}_{+}$, since $\f_{x}$ cannot be constant equal to $0$.

Moreover, for all $x\in\mathbb{R}^\n$ such that $\f(x)\neq 0$, there exists $\lambda>0$ such that $\lambda x\in\level_{\f,x_{0}}$. Indeed, if $\f(x)< \f(x_{0})$, the intermediate value theorem applied to the continuous function $\f_{x_{0}}$ shows that there exists $0 < t < 1$ such that $\f(t x_{0}) = \f_{x_{0}}(t) = \f(x)$, and then $\f(\frac1t x) = \f(x_{0})$. And by interchanging $x$ and $x_{0}$, the same argument holds if $\f(x) > \f(x_{0})$.

The two last paragraphs ensure that for all $x$ such that $\f(x) \neq 0$, there exists a unique positive number $\lambda_{x}$ such that $\lambda_{x} x\in \level_{\f,x_{0}}$. Let us define the function $p$ for all $x\in \R^{\n}$ as follows: if $\f(x) \neq 0$ then $p\del{w}(x)=\frac{1}{\lambda_x}$, otherwise $p(x) = 0$.
We prove in the following that $p$ is \phone.

Let $x\in\mathbb{R}^n$ and $\rho>0$. If $\f(x)=0$ (hence $\f(\rho x) = 0$), then $p(\rho x)= 0 = \rho p(x)$. Otherwise $\f(x)>0$ (hence $\f(\rho x)>0$), and $p(\rho x) = \frac{\rho}{\lambda_{x}}$ since $\frac{\lambda_{x}}{\rho}$ is the (unique) positive number such that $\frac{\lambda_{x}}{\rho} \rho x = \lambda_{x} x \in \level_{\f,x_{0}}$. And thereby $p(\rho x) = \rho \,p(x)$.

We prove that $f = f_{x_{0}} \circ p$, where $f_{x_{0}}$ is a continuous strictly increasing function and $p$ is \phone. Let $x\in\R^{\n}$. If $\f(x)=0$, then $p(x)=0$, and then $\f(x) = 0 = \f(0) = \f_{x_{0}}(0) = (\f_{x_{0}}\circ p) \,(x)$. Otherwise, we have by construction that $\frac{x}{p(x)} \in \level_{\f,x_{0}}$. Therefore $\f(\frac{x}{p(x)}) = \f(x_{0})$ and then $\f(x) = \f(p(x) x_{0}) = \f_{x_{0}}(p(x))$.
By Theorem \ref{homeomorphismTheorem}, $\varphi = \f_{x_{0}}$ is a homeomorphism. Let $\alpha > 0$, define $\tilde{\varphi} = t \mapsto \varphi(t^{1/\alpha})$ and $\tilde{p}= p^{\alpha}$. Then $\tilde{p}$ is \phalpha, $\tilde{\varphi}$ is a homeomorphism and $\f = \tilde{\varphi} \circ \tilde{p}$.

Assume that we have two couples of solutions $(\varphi, p)$ and $(\bar{\varphi}, \bar{p})$ such that $f = \varphi\circ p = \bar{\varphi} \circ \bar{p}$ where $\varphi, \bar{\varphi}$ are homeomorphisms and $p, \bar{p}$ are \phalpha. For all $t > 0$ and $x \in \R^{\n}$, we have for instance $p(t x) = t^{\alpha} p(x)$. Therefore $\image(p) = \R_{+}$. Denote $\psi = \bar{\varphi}^{-1}\circ \varphi$. For all $\lambda > 0$ and $x \in \R^{\n}$, $\psi(\lambda^{\alpha} p(x)) = \psi(p(\lambda x)) = \bar{p}(\lambda x) = \lambda^{\alpha} \psi(p(x))$. Hence $\psi$ is \phone\ on $\R_{+}$. For all $t > 0$, $\psi(t) = t \psi(1)$. Therefore $\psi$ is linear.

{\it Part 2.}
Assume now that there exist $x_{1}, x_{-1} \in \R^{\n}$ such that $\f_{x_{1}}$ is strictly increasing and $\f_{x_{-1}}$ is strictly decreasing. Then $\f(x_{1}) > 0$ and $\f(x_{-1}) < 0$.
Then thanks to the intermediate value theorem, if $\f(x) > 0$, there exists a unique positive number $\lambda_{x}$ such that $\lambda_{x} x \in \level_{\f, x_{1}}$, and if $\f(x) < 0$, there exists a unique positive number $\lambda_{x}$ such that $\lambda_{x} x \in \level_{\f, x_{-1}}$.
We define now $p$ for all $x\in \R^{\n}$ as follows: if $\f(x) = 0$ then $p(x) = 0$, if $\f(x) > 0$ then $p(x) = \frac{1}{\lambda_{x}}$, and finally if $\f(x) < 0$ then $p(x) = -\frac{1}{\lambda_{x}}$.
Let us show that $p$ is \phone. Indeed for any $\rho > 0$ and $x\in\mathbb{R}^n$, if $\f(x)=0$ (hence $\f(\rho x) = 0$), then $p(\rho x)= 0 = \rho p(x)$. If $\f(x)>0$ (hence $\f(\rho x)>0$), \nnew{then}\del{and} $p(\rho x) = \frac{\rho}{\lambda_{x}} = \rho p(x)$ since $\frac{\lambda_{x}}{\rho}$ is the (unique) positive number such that $\frac{\lambda_{x}}{\rho} \rho x = \lambda_{x} x \in \level_{\f,x_{1}}$.
And finally if $\f(x)<0$ (hence $\f(\rho x)<0$), then $p(\rho x) = -\frac{\rho}{\lambda_{x}}=\rho p(x)$ since $\frac{\lambda_{x}}{\rho}$ is the (unique) positive number such that $\frac{\lambda_{x}}{\rho} \rho x = \lambda_{x} x \in \level_{\f,x_{-1}}$. Hence $p$ is \phone.

We define now the function $\varphi : \mathbb{R} \to \mathbb{R}$ such that if $t \geq 0$, $\varphi(t) = \f_{x_{1}}(t)$ and if $t \leq 0$, $\varphi(t) = \f_{x_{-1}}(-t)$.  Then, $\varphi$ is well defined ($\f_{x_{1}}(0) = 0 = f_{x_{-1}}(0)$), continuous and strictly increasing.

Let $x\in\R^{\n}$. If $\f(x)=0$, then $p(x)=0$, and then $\f(x) = 0 = (\varphi \circ p) \,(x)$. If $\f(x)>0$, $\varphi(p(x)) = \f_{x_{1}}(p(x)) = \f(p(x) x_{1}) = \f(x)$ since $\frac{x}{p(x)} \in \level_{\f,x_{1}}$. And finally if $\f(x)<0$, $\varphi(p(x)) = \f_{x_{-1}}(-p(x)) = \f(-p(x) x_{-1}) = \f(x)$ since $-\frac{x}{p(x)} = \lambda_{x}x \in \level_{\f,x_{-1}}$. Thereby, $f = \varphi\circ p$.
Theorem \ref{homeomorphismTheorem} ensures that $\varphi$ is a homeomorphism.
By defining for all $\alpha > 0$, $\tilde{\varphi}(t) = \varphi(t^{1/\alpha})$ if $t \geq 0$, $\tilde{\varphi}(t) = \varphi(-(-t)^{1/\alpha})$ if $t < 0$, $\tilde{p}(x) = p(x)^{\alpha}$ if $p(x) \geq 0$ and $\tilde{p}(x) = - (-p(x))^{\alpha}$ if $p(x) < 0$, it follows that $\f = \tilde{\varphi} \circ \tilde{p}$.

Assume here again that we have two couples of solutions $(\varphi, p)$ and $(\bar{\varphi}, \bar{p})$ such that $f = \varphi \circ p = \bar{\varphi}\circ \bar{p}$ where $\varphi, \bar{\varphi}$ are homeomorphisms and $p, \bar{p}$ are \phalpha. For all $t > 0$ and $x \in \R^{\n}$, we have $p(t x) = t^{\alpha} p(x)$. Therefore $\image(p) = \R$ since $p(x_{1})$ and $p(x_{2})$ have opposite signs. Denote $\psi = \bar{\varphi}^{-1}\circ \varphi$. For all $\lambda > 0$ and $x \in \R^{\n}$, $\psi(\lambda^{\alpha} p(x)) = \psi(p(\lambda x)) = \bar{p}(\lambda x) = \lambda^{\alpha} \psi(p(x))$. Hence $\psi$ is \phone\ on $\R$. For all $t > 0$, $\psi(t) = t \psi(1)$ and $\psi(-t) = t \psi(-1)$ Therefore depending on the values of $\psi(1)$ and $\psi(-1)$, $\psi$ is either linear or piece-wise linear.
\qed
\end{proof}

We now use the previous proposition to prove that a continuous \si\ function is a homeomorphic transformation of a continuous \phone\ function. The proof relies on the result that for a continuous \si\ function, the $f_x$ are either constant or strictly monotonic and continuous (see Proposition~\ref{continuity-monotone}). We distinguish the case where $f$ has a global optimum as in Proposition~\ref{DarbouxGeneral} (i) and the case where $\f$ does not have a global optimum as in Proposition~\ref{DarbouxGeneral} (ii). Overall the following result holds.

\begin{theorem}
Let $\f$ be a continuous \si\ function. Then for all $\alpha > 0$, there exists a continuous \phalpha\ function $p$ and a strictly increasing and continuous bijection (thus a homeomorphism) $\varphi$ such that
$\f = \varphi \circ p$.

For a non-zero $\f$ and $\alpha > 0$, the choice of $(\varphi, p)$ is unique up to a left composition of $p$ with a piece-wise linear function.
If $\f$ admits a global optimum, then $0$ is also a global optimum\del{and for}. For any $x_{0}\in \R^{\n}$ such that $\f(x_{0}) \neq 0$, in the case where $p$ is a {\rm \phone} function, the homeomorphism $\varphi$\del{ corresponding to a {\rm \phone} function} can be chosen as $\f_{x_{0}}$ and is at least as smooth as $\f$.\done{Niko: check/simplify changes}

If $\f$ does not admit a global optimum, then there exist $x_{1}, x_{-1} \in \R^{\n}$ such that $\f(x_{1}) > 0$ and $\f(x_{-1}) < 0$\del{, and for}. For any such $x_{1}$ and $x_{-1}$, in the case where $p$ is a {\rm \phone} function, the homeomorphism $\varphi$\del{ corresponding to a {\rm \phone} function} can
be chosen as the function equal to $\f_{x_{1}}$ on $\R_{+}$  and equal to $t \mapsto \f_{x_{-1}}(-t)$ on $\R_{-}$.
\label{continuousGeneral}
\end{theorem}
\begin{proof}
Let $\f$ be a continuous \si\ function. Thanks to Proposition~\ref{continuity-monotone}, for all $x \in \R^{\n}$, $\f_{x}$ is either constant equal to 0 or strictly monotonic.

{\it Part 1.} Assume that $\f$ has a global optimum. Corollary \ref{localGlobal} shows that $0$ is also a global optimum. 
Then we can apply Proposition~\ref{DarbouxGeneral} in the case where the non-constant $\f_{x}$ have the same monotonicity. Let $x_{0} \in \R^{\n}$ such that $\f(x_{0}) \neq 0$ and define $\varphi = \f_{x_{0}}$. Then $\f = \varphi \circ p$ and $\varphi$ is a homeomorphism. That settles the continuity of the \phone\ function $\varphi^{-1}\circ \f$ thanks to Corollary~\ref{continuous-positively}.

{\it Part 2.} Assume in this part that $\f$ has no global optimum. Since $0$ is not a global optimum, we can find $x_{1}$ and $x_{-1}$ such that $\f(x_{1}) > 0$ and $\f(x_{-1}) < 0$. Therefore $\f_{x_{1}}$ is strictly increasing and $\f_{x_{-1}}$ is strictly decreasing. We apply Proposition~\ref{DarbouxGeneral} in the case where the non-constant $\f_{x}$ do not have the same monotonicity.
If $\varphi$ is the function equal to $\f_{x_{1}}$ on $\R_{+}$ and to $t \mapsto \f_{x_{-1}}(-t)$ on $\R_{-}$, then $\f = \varphi \circ p$ where $\varphi$ is a homeomorphism. That settles the continuity of the \phone\ function $\varphi^{-1}\circ \f$ thanks to Corollary~\ref{continuous-positively}.
For all $\alpha > 0$, the unique construction of $(\varphi, p)$ up to a piece-wise linear function in both parts is a consequence of Proposition~\ref{DarbouxGeneral}. 
\qed \end{proof}

\subsection{Sufficient and Necessary Condition to be the Composite of a \ph\ Function}
\label{sec:iff}
We have seen in the previous section that a \emph{continuous} \si\ function can be written as $\varphi \circ p$ with $\varphi$ strictly monotonic and $p$ \ph. Relaxing continuity, we prove in the next theorem some necessary \emph{and} sufficient condition under which an \si\ function is the composite of a \ph\  function with  a strictly monotonic function.

\begin{theorem}

Let $\f$ be an \si\ function. There exist a \phone\ function $p$ and a strictly increasing function\del{ and continuous bijection (thus a homeomorphism)} $\varphi$ such that $\f = \varphi \circ p$ if and only if for all $x\in\mathbb{R}^n$, the function $\f_x$ is either constant or  strictly monotonic and the strictly increasing $\f_x$ share the same image (\ie, if $\lambda\in\mathbb{R}$ is reached for one of these functions, then it is reached for all of them) and the strictly decreasing ones too.

For a non-zero $\f$, up to a left composition of $p$ with a piece-wise linear function, the choice of $(\varphi, p)$ is unique.
\label{equivalence-transformation}
\end{theorem}

\begin{proof}
We prove first the forward implication.
Suppose there is a \phone\ function $p$ and a strictly monotonic function $\varphi$ such that $f=\varphi\circ p$. Consider $x\in\mathbb{R}^n$. Either $p(x)=0$ and then for any $t\geqslant 0$ we have that $p(tx)=0$, so that $\f_x(t) = f(tx)=\varphi(p(tx)) = \varphi(0)$ and $\f_{x}$ is constant on $\mathbb{R}_+$, or $p(x)\neq 0$, and then $t\in\mathbb{R}_+\mapsto p(tx)=tp(x)$ is strictly monotonic, and $f_x(t)=\varphi(p(tx))$ is strictly monotonic too on $\mathbb{R}_+$. Moreover, consider $x_1\neq x_2$ such that $f_{x_1}$ and $f_{x_2}$ are increasing. Then $p(x_1)$ and $p(x_2)$ are of the same sign, so there is some $t_*>0$ such that $p(x_1)=t_*p(x_2)=p(t_*x_2)$, so the functions $t\mapsto f(tx_1)$ and $t\mapsto f(tt_*x_2)$ are equal, so the functions $f_{x_1}$ and $f_{x_2}$ take the same values. The same applies on the strictly decreasing functions.

We now prove the backward implication.
Suppose that the functions $f_x$ are either constant or  strictly monotonic and the increasing ones share the same values and the decreasing ones too.

If all the $\f_{x}$ are constant, then for all $x\in\R^{\n}$, $\f(x) = \f(0) = 0$ and it is enough to write $\f = \varphi \circ p$ with $p = t \mapsto t$ on $\R_{+}$ and $p=0$. We assume from now on that at least one $\f_{x}$ is not constant.

Consider that all the non-constant $\f_{x}$ have the same monotonicity. Let us choose $x_{0}$ such that $\f(x_{0}) \neq 0$. Then for all $x \neq 0$, $f_{x}$ and $f_{x_{0}}$ have the same monotonicity. Since they have the same image and are injective, there exists a unique $\lambda_{x} > 0$ such that $\lambda_{x} x\in \level_{\f,x_{0}}$. We then define $p$ and $\varphi$ as in the Part 1 of Theorem~\ref{continuousGeneral} to ensure that $\f = \varphi \circ p$ where $p$ is \phone\ and $\varphi$ is strictly monotonic.

Consider finally that all the non-constant $\f_{x}$ do not have the same monotonicity. Then there exist $x_1$ and $x_{-1}$ such that $f(x_1)>0$ and $f(x_{-1})<0$.  
Then, thanks to the assumption that all increasing $f_x$ share the same values and the strictly decreasing $\f_x$ too, if $\f(x)>0$, then there exists a unique positive number $\lambda_x$ such that $\lambda_x x\in\mathcal{L}_{f,x_1} = \acco{y\in\R^{\n}, \f(y) = \f(x_{1})}$, and if $f(x)<0$, then there exists a unique (thanks to the assumption of strict monotonicity for the non-constant $f_x$) positive number $\lambda_x$ such that $\lambda_xx\in\mathcal{L}_{f,x_{-1}}$. Therefore, we can define $p$ as in Theorem~\ref{continuousGeneral}. As before, $p$ is \phone. 
Define also the function $\varphi:\R\to\R$ as in Theorem~\ref{continuousGeneral}. It is still increasing, but not necessarily continuous.
Then, as in Theorem~\ref{continuousGeneral}, $\f=\varphi\circ p$.

The proof of the unicity of $\pare{\varphi, p}$ up to a piece-wise real linear function is similar to the proof in Proposition~\ref{DarbouxGeneral}.
\qed \end{proof}
Complementing Theorem~\ref{equivalence-transformation}, we construct an example of an \si\ function that can not be decomposed as $\f={}\varphi\circ p$, because the strictly increasing $\f_{x}$ do not share the same image. Define $\f$ such that for all $x \in \R^{\n}$, $\f(x) = \tanh(x_{1})$ if the first coordinate $x_{1} \geq 0$ and $\f(x) = 1 + \exp(-x_{1})$ otherwise. Then $\f$ is \si\ and if $x_{1} \neq 0$, $\f_{x}$ is strictly increasing. However for all $x$ such that $x_{1} > 0$ then $\image(\f_{x}) = [0, 1)$ and otherwise for $x$ such that $x_1 \leq 0$, $\image(\f_{x}) = \acco{0}\cup (2, \infty)$.

\section{Level Sets of \si\ Functions}
\label{section-level-sets}

Scaling-invariant functions appear naturally when studying the convergence of comparison-based optimization algorithms \cite{auger2016linear}. In this specific context, the difficulty of a problem is entirely determined by its level sets whose properties are studied in this section.

\subsection{Identical Sublevel Sets}
\label{equivalence-positive}

Level sets and sublevel sets of a function $f$ remain unchanged if we compose the function with a strictly increasing function $\varphi$ since
\begin{align}
\f(x) \leq \f(y) \iff \varphi (f(x)) \leq \varphi(f(y))
\enspace.
\label{invariance-strict-increasing}
\end{align}

We prove in the next theorem that two arbitrary functions $\f$ and $p$ have the same level sets if and only if $f = \varphi \circ p$ where $\varphi$ is strictly increasing.

\begin{theorem}
Two functions $\f$ and $p$ have the same sublevel sets if and only if there exists a strictly increasing function $\varphi$ such that $\f = \varphi \circ p$.
\label{samesublevelsets}
\end{theorem}
\begin{proof}
If $f = \varphi \circ p$ with $\varphi$ strictly increasing, since sublevel sets are invariant by $\varphi$, $f$ and $p$ have the same sublevel sets.
Now assume that $\f$ and $p$ have the same sublevel sets.
Then for all $x \in \R^{\n}$, there exists $T(x) \in \R^{\n}$ such that $\sublevel_{\f, x} = \sublevel_{p, T(x)}$. In other words\nnew{,} for all $y \in \R^{\n}$,  
$ \f(y) \leq \f(x) \iff p(y) \leq p(T(x))$.
We define the function $$\phi: \left\{\fun{{\rm Im}(\f)}{{\rm Im}(p)}{\f(x)}{ p\pare{ T(x) } }\right. \enspace. $$ The function
$\phi$ is well-defined because for $x, y \in \R^{\n}$ such that $\f(x) = \f(y)$, $\sublevel_{\f,x} = \sublevel_{\f,y}$. And since $\sublevel_{\f,x} = \sublevel_{p,T(x)}$ and $\sublevel_{\f,y} = \sublevel_{p, T(y)}$, then $\sublevel_{p, T(x)} = \sublevel_{p, T(y)}$, and then $p\pare{ T(x)} = p\pare{ T(y)}$. Therefore $\phi(\f(x)) = \phi(\f(y))$. By construction we have that
$ \phi\circ \f = p \circ T$.

Let us show that $p\circ T = p$. We have  $\sublevel_{p\circ T, x} = \sublevel_{\f, x} = \sublevel_{p, T(x)}$. Then $x \in \sublevel_{p, T(x)}$ and then $p(x) \leq p(T(x))$. Therefore $p \leq p\circ T$. In addition for all $y \in \R^{n}$, there exists $x$ such that $\sublevel_{p, y} = \sublevel_{\f, x} = \sublevel_{p, T(x)}  =  \sublevel_{p\circ T, x} $. Then $y \in \sublevel_{p\circ T, x}$, which induces that $p(T(y)) \leq p(T(x))$. Plus, $\sublevel_{p, y} = \sublevel_{p, T(x)}$, therefore $p(T(x)) = p(y)$. Thereby $p(T(y)) \leq p(y)$, and then $p\circ T \leq p$. Finally $p\circ T = p$. Hence $\phi\circ \f = p$.

Let us prove now that $\phi$ is strictly increasing. Consider $x, y \in \R^{\n}$ such that $\f(x) < \f(y)$. Then $\sublevel_{f, x} \subset \sublevel_{f, y}$ with a strict inclusion, which means that $\sublevel_{p, T(x)} \subset \sublevel_{p, T(y)}$ with a strict inclusion. Thereby $p \pare{ T(x) } < p \pare{ T(y)}$, \ie, $\phi (\f (x)) < \phi (\f (y))$. Hence $\phi$ is strictly increasing.
And up to restricting $\phi$ to its image, we can assume without loss of generality that $\phi$ is a strictly increasing bijection. We finally denote $\varphi = \phi^{-1}$ and it follows that $\f = \varphi \circ p$. \qed 
\end{proof}

Theorem~\ref{samesublevelsets} and Theorem~\ref{equivalence-transformation} give both equivalence conditions for an \si\ function $\f$ to be equal to $\varphi\circ p$ where $\varphi$ is strictly increasing and $p$ is positively homogeneous\footnote{Note that in Theorem~\ref{equivalence-transformation}, we can assume without loss of generality that $\varphi$ is always strictly increasing by replacing if needed $\varphi$ by $t \to \varphi(-t)$ and $p$ by $-p$.}. One condition is that there exists a \ph\ function with the same sublevel sets as $\f$, while the other condition is that the $\f_{x}$ are either constant or  strictly monotonic, and the strictly increasing and decreasing ones have the same image, respectively.

\subsection{Compactness of the Sublevel Sets}
\label{section-compact}
Compactness of sublevel sets is relevant for analyzing step-size adaptive randomized search algorithms \cite{auger2013linear, morinaga2019generalized}. We investigate here how compactness properties shown for positively homogeneous functions extend to scaling-invariant functions.
For an \si\ function $\f$, we have $\sublevel_{f, t x} = t \sublevel_{f,x}$.
When\del{Since} $\psi: \foncsmall{y}{t y}$ is a homeomorphism, we have that $\psi(\sublevel_{f,x})$ equals $t \sublevel_{f,x}$ and is compact if and only if $\sublevel_{f,x}$ is compact. Therefore, for all $t > 0$:
\begin{align}
\sublevel_{f,tx}  \text{ is compact if and only if } \sublevel_{f,x}   \text{ is compact. } 
\label{compact-lines}
\end{align}

Furthermore, if $p$ is a lower semi-continuous positively homogeneous function such that $p(x) > 0$ for all nonzero $x$ then the sublevel sets of $p$ are compact~\cite[Lemma 2.7]{auger2013linear}.  We recall it with all the details in the following proposition:

\begin{proposition}[\nnew{Lemma 2.7 in} \cite{auger2013linear}]
Let $p$ be a positively homogeneous function with degree $\alpha>0$ and $p(x)>0$ for all $x\neq 0$ (or equivalently\nnew{,} $0$ is the unique global argmin of $p$) and $p(x)$ finite for every $x\in\R^n$. Then for every $x \in \R^{\n}$, the following holds:
\begin{enumerate}
\item[{\rm (i)}] $\dsp\lim_{t\to 0} p(tx)=0$ and for all $x\neq 0$ the function
$p_{x}: \foncfast{ [ 0, \infty ) }{t}{p(t x)}{\R_{+}}$
 is continuous, strictly increasing and converges to $\infty$ when $t$ goes to $\infty$.
\item[{\rm (ii)}] If $p$ is lower semi-continuous, the sublevel set $\sublevel_{p, x}$ is compact.
\end{enumerate}
\end{proposition}

We prove a similar theorem for lower semi-continuous \si\ functions $\f$ with continuous $\f_{x}$ functions, showing in particular that the unicity of the global argmin is equivalent to \del{the above }items \nnew{similar to the above}.
Note that we need to assume the continuity of the functions $f_x$, while this property is unconditionally satisfied for positively homogeneous functions where $p_{x}(t) = t^{\alpha}p_{x}(1)$ for all $x \in \R^{\n}$ and for all $t > 0$.

\begin{theorem}
Let \f\ be \si. Then the conditions
\begin{itemize}
\item $\f(x) > 0$ for all $x \not= 0$ and
\item $0$ is the unique global argmin
\end{itemize}
are equivalent.
Let $\f$ be additionally lower semi-continuous and for all $x \in \R^{\n}$, $\f_{x}$ is continuous on a neighborhood of $0$. Then the following are equivalent:
\begin{itemize}
\item[{\rm (i)}] $0$ is the unique global argmin.

\item[{\rm (ii)}] \nnew{F}or any $x\in \R^{\n}\backslash\acco{0}$, the function $\f_{x}$ is strictly increasing.

\item[{\rm (iii)}] The sublevel sets $\sublevel_{\f, x}$ for all $x$ are compact.
\end{itemize}
\label{equivalences}
\end{theorem}

\begin{proof} 
Since, w.l.o.g., \f\ is given such that $\f(0)=0$, its unique global argmin is $0$ if and only if $\f(x) > 0$ for all $x\not=0$.
Now we prove first that \emph{{\rm (i)} $\Rightarrow$ {\rm (ii)}:} Let $x \in \R^{\n}\backslash\acco{0}$. Assume (by contraposition) that there exists $0 < t_{1} < t_{2}$ such that $\f_{x}(t_{1}) = \f_{x}(t_{2})$. Then by scaling-invariance, $\f_{x}(1) = \f_{x}\left(\frac{t_{1}}{t_{2}}\right)$. It follows by multiplying iteratively by $\frac{t_{1}}{t_{2}}$ that for all $k\in\mathbb{Z}_{+}$,
$\f_{x}(1) = \f_{x}\pare{\left(\frac{t_{1}}{t_{2}}\right)^{k}}$. Therefore if we take the limit when $k \to \infty$, it follows thanks to the continuity of $\f_{x}$ at $0$ that $\f( x) = \f_{x}(1) = \f(0)$, which contradicts the assumption ${\rm (i)}$. Hence, $\f_{x}$ is an injective function. Plus, there exists $\epsilon > 0$ such that $\f_{x}$ is continuous on $[0, \epsilon]$. Therefore $\f_{x}$ is an injective continuous function on $[0, \epsilon]$, which implies that $\f_{x}$ is a strictly monotonic function on $[0, \epsilon]$. Lemma~\ref{lemma-monotone} implies therefore that $\f_{x}$ is strictly monotonic. And since $0$ is an argmin of $\f_{x}$, then $\f_{x}$ is strictly increasing.

\emph{{\rm (ii)} $\Rightarrow$ {\rm (iii)}:} $\f$ is lower semi-continuous on the compact $\S_{1}$, then it reaches its minimum on that compact: there exists $s \in \S_{1}$ such that $ \f(s) = \dsp\min_{z\in\S_{1}} \f(z)$. Also, since sublevel sets of lower semi-continuous functions are closed, then $\sublevel_{\f,s}$ is closed. Now let us show that it is also bounded.

 If $y \in \sublevel_{\f,s}\backslash\acco{0}$, then $\f(y) \leq \f(s) \leq \f \left( \frac{y}{\|y\| } \right)$. And since $\f_{y}$ is strictly increasing, we obtain that $1 \leq \frac{1}{\|y\| }$, thereby $\|y\| \leq 1$. We have shown that $\sublevel_{\f,s} \subset \overline{\B}_{1}$. 
 
 Then $\sublevel_{\f,s}$ is a compact set, as it is a closed and bounded subset of $\R^{\n}$. By \eqref{compact-lines}, it follows that $\sublevel_{\f, t s}$ is compact for all $t > 0$. 

For all $x\in\R^{\n}\backslash{\acco{0}}$, $\f(x) > \f(0)$ thanks to ${\rm (ii)}$. Then $\sublevel_{\f,0} = \acco{0}$ and is compact. 

Let $x\in\R^{\n}\backslash\acco{0}$. Then there exists $\epsilon > 0$ such that $\f_{\frac{x}{\|x\| } }$ is continuous on $[0, \epsilon]$. We have that $\f_{\frac{x}{\|x\| } }(0) = 0 < \f(s) \leq \f_{\frac{x}{ \|x\| } }(1)$, then by scaling-invariance, $\f_{\frac{x}{ \|x\| } }(0) <  \f(\epsilon s) \leq  \f_{\frac{x}{ \|x\| } }(\epsilon)$. Therefore by the intermediate value theorem applied to $\f_{\frac{x}{\|x\| } }$ continuous on $[0, \epsilon]$, there exists $t \in (0, \epsilon]$ such that $f_{\frac{x}{\|x\| }}(t) = \f(\epsilon s)$.
Then $\sublevel_{\f, t\frac{x}{\|x\| }} = \sublevel_{\f, \epsilon s}$ and is compact. We apply again \eqref{compact-lines} to observe that $\sublevel_{\f, x}$ is compact.

\emph{{\rm (iii)} $\Rightarrow$ {\rm (i)}:} Let $x\in\sublevel_{\f, 0}$, then $t x\in\sublevel_{\f, 0}$ for all $t \geq 0$, and then $\acco{t x}_{t \in \R_{+}} \subset \sublevel_{\f, 0}$ which is a compact set. This is only possible if $x = 0$, otherwise the set $\acco{t x}_{t \in \R_{+}}$ would not be bounded. Hence $\sublevel_{\f, 0} = \acco{0}$ which implies that $0$ is the unique global argmin.
\qed \end{proof}
We derive from Theorem~\ref{equivalences} the next corollary, stating\del{ that for a continuously differentiable} \del{when }for a lower semi-continuous \si\ function with a unique global argmin\nnew{, when} the intersection of any half-line of origin $0$ and a level set is a singleton.
\begin{corollary} Let $\f$ be a lower semi-continuous\del{continuously differentiable} \si\ function with $0$ as unique global argmin.
Assume that for all $x\in \R^{\n}$, $\f_{x}$ is continuous on a neighborhood of $0$, and the $\f_{x}$ share the same image.
Then for all $x \in \R^{\n}$, any half-line of origin $0$ intersects $\level_{\f, x}$ at a unique point.
\label{unique-intersection-level-sets}
\end{corollary}
\begin{proof}
For all non-zero $x$, Theorem~\ref{equivalences} ensures that $\f_{x}$ is strictly increasing. Therefore for a non-zero $x$, $\f_{x}$ is injective. And then the intersection of a level set and a half-line of origin $0$ contains at most one point. In addition the $\f_{x}$ share the same image for all non-zero $x$. Then for two non-zero vectors $x, y$, there exists $t \geq 0$ such that $\f_{y}(t) = \f_x(1)$. In other words, there exists $t \geq 0$ such that $t y \in \level_{\f, x}$. We end this proof by noticing that $\level_{\f, 0} = \acco{0}$ and then intersects any half-line of origin 0 only at $0$.
\qed
\end{proof}
\subsection{Sufficient Condition for Lebesgue Negligible Level Sets}
\label{negligible-level-sets}

We assume that $\f$ is lower semi-continuous \si\ admitting a unique global argmin and all $f_x$ are continuous and prove that $\f$ has  Lebesgue negligible level sets.

\begin{proposition}
Let $\f$ be an \si\ function with $0$ as unique global argmin. 
Assume\del{also} that $\f$ is lower semi-continuous and for all $x\in \R^{\n}$, $\f_{x}$ is continuous.
Then the level sets of $\f$ are Lebesgue negligible.
\label{zeroLevel}
\end{proposition}
\begin{proof} Let $x \in \R^{\n}$. Let us denote by $\mu$ the Lebesgue measure. For all $t > 0$, $\mu\pare{\mathcal{L}_{\f,tx}} = \mu\pare{t \mathcal{L}_{\f,x}} = t^{\n}\mu\pare{\mathcal{L}_{\f,x}}$, thanks to \eqref{homothetic}. Therefore, 
$\text{ if $t \geq 1$, } \mu\pare{\mathcal{L}_{\f,tx}} \geq \mu\pare{\mathcal{L}_{\f,x}}. $
In addition, for all $k \geq 1$,
$
\dsp \mathcal{L}_{\f,(1+\frac{1}{k})x} \subset \acco{ y \in \R^{\n}, \f(x) \leq \f(y) \leq \f(2x)} \subset \sublevel_{\f,2x}\del{,}
$
because if $x \neq 0$, $\f_{x}$ is strictly increasing thanks to Theorem \ref{equivalences}. And the same theorem induces that $\sublevel_{\f,x}$ is compact and hence $\mu\pare{\sublevel_{\f,x}} < \infty $. It follows that
$ 
\dsp \sum_{k=1}^{\infty} \mu\pare{ \mathcal{L}_{\f,x} } \leq \sum_{k=1}^{\infty} \mu\pare{ \mathcal{L}_{\f,(1+\frac{1}{k})x} } \leq \mu\pare{ \sublevel_{\f,2x} } < \infty.
$
Hence, $\mu\pare{\mathcal{L}_{\f,x} } = 0$.
\qed \end{proof}

\subsection{Balls Containing and Balls Contained in Sublevel Sets}
\label{elitist-es}

The sublevel sets  of continuous \ph\ functions include and are embedded in balls whose construction is scaling-invariant. Given that continuous \si\ functions are monotonic transformation of \ph\ functions, those properties are naturally transferred to \si\ functions. This is what we formalize in this section.

From the definition of a \ph\ function with degree $\alpha$, for all $x \neq 0$ we have 
$p(x) = \dsp \| x \|^{\alpha} p\pare{{x}/{ \| x \| } }$ for all $x \neq 0$.
Therefore, $p$ is continuous on $\R^{\n}\setminus \acco{0}$ if and only if $p$ is continuous on $\S_{1}$. For such $p$, we denote
$
m_{p} = \dsp \min_{x\in \S_{1}} p(x),
$
and
$
M_{p} = \dsp \max_{x\in \S_{1}} p(x).
$
We have the following propositions:

\begin{proposition}[\nnew{Lemma 2.8 in \cite{auger2013linear}}]
Let $p$ be a \ph\ function with degree $\alpha$ such that $p(x) > 0$ for all $x \neq 0$. Assume that $p$ is continuous on $\S_{1}$, then for all $x \neq 0$, the following holds
\begin{align}
\| x \| m_{p}^{1/\alpha} \leq p(x)^{1/\alpha} \leq \| x \| M_{p}^{1/\alpha} \enspace.
\end{align}
\label{boundpositive}
\end{proposition}

\begin{proposition}[\nnew{Lemma 2.9 in \cite{auger2013linear}}]
Let $p$ be a \ph\ function with degree $\alpha$ such that $g(x) > 0$ for all $x \neq 0$. Assume that $p$ is continuous on $\S_{1}$.
Then for all $\rho > 0$, the ball centered in $0$ and of radius $\rho$ is included in the sublevel set of degree $\rho^{\alpha} M_{p}$, \ie,
$
\B\pare{0, \rho} \subset \sublevel_{p, \rho x_{M_{p}}}, \text{ with } p(x_{M_{p}}) = M_{p}. 
$
For all $x \neq 0$, the sublevel set of degree $p(x)$ is included into the ball centered in $0$ and of radius $(p(x)/m_{p})^{\alpha}$, \ie,
\begin{equation*}
\sublevel_{p, x } \subset \B\pare{0, \pare{\frac{p(x)}{m_{p}}}^{\alpha}} .
\end{equation*}
\label{ballincluded}
\end{proposition}
We can generalize both propositions to continuous scaling-invariant functions using Theorem~\ref{continuousGeneral}.

\begin{proposition}
Let $\f$ be a continuous \si\ function such that $\f(x) > 0$ for all $x \neq 0$. Then there exist an increasing homeomorphism $\varphi$ on $\R_{+}$ and two positive numbers $0< m \leq M$ such that
\begin{itemize}

\item[{\rm (i)}] for all $x \neq 0$, $
\varphi\pare{ m\|x\|} \leq \f(x) \leq \varphi\pare{M\|x\|}\del{,}$\nnew{;}
\item[{\rm (ii)}] for all $\rho > 0$, the ball centered in $0$ and of radius $\rho$ is included in the sublevel set of degree $\varphi( \rho \varphi^{-1}(M))$, \ie, 
$
\B\pare{0, \rho} \subset \sublevel_{\f, \rho x_{M}}$ with $\f(\rho x_{M}) = \varphi \pare{ \rho \varphi^{-1}(M) }$\del{.}\nnew{;}

\item[{\rm (iii)}] for all $x \neq 0$, the sublevel set of degree $\f(x)$ is included into the ball centered in $0$ and of radius $\dsp\frac{\varphi^{-1}\pare{\f(x)}}{\varphi^{-1}(m)}$, \ie, 
\begin{align}
\sublevel_{\f, x } \subset \B\pare{0, \frac{\varphi^{-1}\pare{\f(x)} }{\varphi^{-1}(m)}}
\enspace.
\end{align}
\end{itemize}
\label{boundballincludedscaling}
\end{proposition}
\begin{proof}
Thanks to Theorem~\ref{continuousGeneral}, we can write $\f = \f_{x_{0}}\circ p$ where $x_{0} \neq 0$, $p$ \nnew{is} \phone\ \nnew{and} $\varphi$ defined as $\f_{x_{0}}$ is an increasing homeomorphism. Then $p = \varphi^{-1}\circ \f$ and hence verifies: for all $x \neq 0$, $p(x) > 0$.
Define\del{ $m$ and $M$ as}
$
m = \varphi(m_{p}) \text{ and } M = \varphi(M_{p}),
$
where $m_{p} = \dsp \min_{x\in \S_{1}} p(x)$ and $M_{p} = \dsp \max_{x\in \S_{1}} p(x)$.

For $x \neq 0$, Proposition \ref{boundpositive} ensures that
$
m_{p} \| x \| \leq p(x) \leq M_{p} \| x \|.
$
Taking the image of this equation with respect to $\varphi$ proves ${\rm (i)}$.

For all $\rho > 0$,  
$\B\pare{0, \rho} \subset \sublevel_{p, \rho x_{M_{p}}}$, with $p(x_{M_{p}}) = M_{p}$. 
Since sublevel sets are invariant with respect to an increasing bijection, it follows that $\sublevel_{p, \rho x_{M_{p}}} = \sublevel_{\f, \rho x_{M_{p}}}$. In addition,
$\f(\rho x_{M_{p}}) =  \varphi(p(\rho x_{M_{p}})) = \varphi(\rho p(x_{M_{p}})) = \varphi( \rho \varphi^{-1}(M))$ such that we have proven ${\rm (ii)}$.

Let $x \neq 0$. Again by invariance of the sublevel set, $\sublevel_{p, x} = \sublevel_{\f, x}$. And Proposition \ref{ballincluded} says that 
$
\sublevel_{p, x } \subset \B\pare{0, \frac{p(x)}{m_p}}$.
We obtain the results with the facts that $p = \varphi^{-1}\circ \f$ and $m_{p} = \varphi^{-1}(m)$.
\qed \end{proof}

\subsection{A Generalization of a Weak Formulation of Euler's Homogeneous Function Theorem}
\label{generalEuler}

For a function $p: \R^{\n} \to \R$ continuously differentiable on  $\R^{\n}\setminus \acco{0}$, Euler's homogeneous function theorem states that there is equivalence between $p$ is \ph\ with degree $\alpha$ and for all $x \neq 0$
\begin{equation}\label{eq:euler}
\alpha p(x) = \nabla p(x) \cdot x
\enspace.
\end{equation}
If in addition $p$ is continuously differentiable in zero\del{on $\R^{\n}$}, then $\alpha p(0) = 0 = \nabla p(0) \cdot 0$.
Along with~\eqref{eq:euler}, this latter equation implies that at each point $y$ of a level set $\level_{p,x}$, the scalar product between $\nabla p (y)$ and $y$ is constant equal to $\nabla p(x) \cdot x$ or that the level sets of $p$ and of the function $x \mapsto \nabla p(x) \cdot x$ are the same, that is, the level sets of a continuously differentiable \ph\ function satisfy
\begin{align}
\level_{p, x} = \level_{z \mapsto \nabla p(z) \cdot z,x} = \acco{y \in \R^{\n}, \nabla p(y) \cdot y = \nabla p(x) \cdot x }
\enspace.
\label{levelset-homogeneous}
\end{align}
We call this a weak formulation of Euler's homogeneous function theorem.

If $f$ is a continuous \si\ function, we can write $f$ as $\varphi \circ p$ where $p$ is \ph\ and $\varphi$ is a homeomorphism, according to Theorem~\ref{continuousGeneral}. We have the following proposition in the case where $\varphi$ and $p$ are also continuously differentiable.

\begin{proposition}
Let $\f$ be a continuously differentiable \si\ function that can be written as $\varphi\circ p$ where $p$ is $\ph_{\alpha}$, $\varphi$ is a homeomorphism, and $\varphi$ and $\varphi^{-1}$ are continuously differentiable (and thus $p$ is continuously differentiable). Then for all $x\in \R^{\n}$,
\begin{align}
\nabla f(x) \cdot x = \alpha \,\varphi^{\prime}(p(x)) \, p(x)
\enspace.
 \end{align}
 \label{general-euler-si}
\end{proposition}

\begin{proof}
Since $p = \varphi^{-1}\circ \f$, it is continuously differentiable. From the chain rule, for all $x \in \R^{\n}:$
$\nabla f(x) \cdot x = \varphi^{\prime}(p(x)) \nabla p(x) \cdot x = \alpha \varphi^{\prime}(p(x)) p(x)$.
The last equality results from the Euler's homogeneous theorem applied to $p$.
\qed \end{proof}

Yet, the assumptions of the previous proposition are not necessarily satisfied when $\f$ is a continuously differentiable SI function. Indeed, we exhibit in the next proposition an example of a\nnew{n} SI and continuously differentiable function $\f$ such that $\f =\varphi\circ p$ but either $p$ or $\varphi$ is non-differentiable.
\begin{proposition}
Define $\varphi: t \mapsto \dsp\int_{0}^{t} \frac{1}{1+\log^{2}(u)} \diff u$ on $\R_{+}$ and $p: x \mapsto \abs{x_{1}}$. Then $\f = \varphi \circ p$ is continuously differentiable and \si. 
Yet, for any $\tilde{\varphi}$ strictly increasing and $\tilde{p}$ PH such that $f = \tilde{\varphi} \circ \tilde{p}$ (including $\varphi$ and $p$ above), 
either $\tilde{p}$ is not differentiable on any point of the set $\acco{x; \,x_{1} = 0}$ or $\tilde{\varphi}$ is not differentiable at $0$.
\label{counterexample}
 \end{proposition} 
\begin{proof}
Let us prove that $\f$ is continuously differentiable. For $x \neq 0$, $\nabla f(x) = \frac{1}{1+\log^{2}( \abs{x_{1}}  )}  \frac{x_{1}}{ \abs{x_{1}}} e_{1}$ where $e_{1}$ is the unit vector $\pare{1, 0, \dots, 0}$.
Then $\dsp\lim_{x\to 0} \nabla f(x)$ exists and is equal to $0$, hence $\f$ is continuously differentiable.

Assume that $\pare{\tilde{\varphi}, \tilde{p}}$ is such that $\varphi\circ p = \tilde{\varphi}\circ \tilde{p}$, with $\tilde{\varphi}$ strictly increasing and $\tilde{p}$ $\text{\rm PH}_{\alpha}$. Denote $\psi = \tilde{\varphi}^{-1}\circ \varphi$. For all $\lambda > 0$ and $x \in \R^{\n}$, $\psi(\lambda p(x)) = \psi(p(\lambda x)) = \tilde{p}(\lambda x) = \lambda^{\alpha} \psi(p(x))$. Therefore $\psi$ is $\text{\rm PH}_{\alpha}$ on $\image(p) = \R_{+}$, hence for all $t > 0$, $\psi(t) = t^{\alpha} \psi(1)$. Then up to a positive constant multiplicative factor, $\tilde{p}(x) = \abs{x_{1}}^{\alpha}$ and $\tilde{\varphi}(t) = \varphi(t^{1/\alpha})$. And then if $\tilde{p}$ is differentiable, we necessarily have that $\alpha > 1$. 

In the case where $\alpha > 1$, for all $t > 0$, $\tilde{\varphi}^{\prime}(t) = \frac{1}{\alpha} \frac{t^{\frac{1}{\alpha} - 1}}{1 + \log^{2} (t^{1/\alpha})} $ and then $\tilde{\varphi}$ is not differentiable at $0$. \qed
\end{proof}

Yet we can prove\del{ that} for all continuously differentiable \si\ functions, \nnew{$f$, that} the level set of $f$ going through $x$, \ie, $\level_{f,x}$\nnew{,} is included in the level set of $z \mapsto \nabla f(z) \cdot z$ going through $x$.
\begin{lemma}
For a continuously differentiable \si\ function $\f$ and for $x \in \R^{\n}$,
\begin{align}
\mathcal{L}_{f,x} \subset \level_{z\mapsto \nabla f(z) \cdot z, x} = \acco{y \in \R^{\n}, \nabla \f(y) \cdot y = \nabla \f(x) \cdot x }.
\label{level-sets-vs-gradients}
\end{align}
That is, each level set of $f$ has a single value of $\nabla \f(x) \cdot x$
while also different level sets of $f$ can have the same value of $\nabla \f(x) \cdot x$.
\label{lemma-level-sets-vs-gradients}%
\end{lemma}
\begin{proof}
Let $y \in \level_{\f, x}$. Since $\f(y) = \f(x)$, then for all $t \geq 0$, $\f(ty) = \f(tx)$. We define the function $h$ on $\R_{+}$ such that for all $t \geq 0$, $h(t) = \f(tx) - \f(ty)$. Then $h$ is the zero function, so is its derivative: $h^{\prime}(t) = \nabla \f(tx) \cdot x - \nabla \f(ty) \cdot y = 0$ for all $t \geq 0$.
In particular we have the result for $t = 1$.
\qed \end{proof}
 
We exhibit in the next proposition a continuously differentiable \si\ function where the inclusion in the above lemma is strict (another example is Lemma~\ref{lemma:saddle}).

\begin{proposition}
Let $p$ be the \ph$_2$ function $x \in \R^{\n} \mapsto \| x \|^{2}$ and $\varphi$ the strictly monotonic function $\varphi(t) = \exp(-t)$ for all $t \geq 0$. Then $\f: x \mapsto \varphi (p (x)) =  \exp( - \| x \|^2)$ is continuously differentiable. For any $0 < r < 1$, there is a unique $s > 1$ such that for any $x \in \S_{r}:$
\begin{align*}
\level_{z\mapsto \nabla f(z) \cdot z, x} = \level_{f,x}\cup \level_{f,\frac{s}{r} x } 
\end{align*}
where $\level_{f,x}$ and $ \level_{f,\frac{s}{r} x } $ are disjoint.
\end{proposition}

\begin{proof}
 Remark that the transformation could be chosen to obtain a degree equal to $1$ as in Theorem~\ref{continuousGeneral}, but the differentiability of $p$ would not be guaranteed. 

We notice that $t \to t \varphi^{\prime}(t)$ is not injective on $\R_{+}$. It is injective on $[0, 1)$ and on $[1, \infty)$ and for any $0 < r < 1$ there is a unique $s > 1$ such that 
 \begin{align}
r^{2} \varphi^{\prime}(r^{2}) = s^{2} \varphi^{\prime}(s^{2}). 
 \label{non-injectivity}
 \end{align}
 We will prove that for any such $r, s$ and any $x \in \S_{r}$,
\begin{align}
 \acco{ y \in \R^{\n}, \nabla f(y)\cdot y = \nabla f(x)\cdot x } = \level_{f,x}\cup \level_{f,\frac{s}{r} x } 
 \label{counter-example-levelset-gradient}
\end{align}
Let $x \in \R^{\n}$ such that $\| x \| = r$. By the chain rule,  for all $y \in \R^{\n}$ we have
$$
\nabla f(y) \cdot y = \varphi^{\prime}(p(y)) \nabla p(y) \cdot y =  2 \varphi^{\prime}(p(y)) p(y)\enspace.
$$
Therefore $y \in \acco{ y \in \R^{\n}, \nabla f(y)\cdot y = \nabla f(x)\cdot x }$ if and only if $\|y\|^{2} \varphi^{\prime}(\|y\|^{2}) = r^{2} \varphi^{\prime}(r^{2}) $.
From \eqref{non-injectivity}, we know that this is possible only if $\|y\| = r$ or $\|y\| = s$, \ie, only if $\f(y) = \f(x)$ or $\f(y) = \f(\frac{s}{r}x)$.
Hence the equality in \eqref{counter-example-levelset-gradient}.

We remark that $\level_{f,x}$ and $\level_{\f,\frac{s}{r} x}$ are disjoint whenever $\f(x) \neq \f(\frac{s}{r} x)$. If $x \in \S_{r}$, $\f(x) = e^{-r^{2}} \neq e^{-s^{2} } = \f\pare{\frac{s}{r} x}$ which implies that $\level_{f,x}$ and $\level_{\f,\frac{s}{r} x}$ are disjoint.
\qed \end{proof}

The non-injectivity of $t \to t \varphi^{\prime}(t)$ is essential in the above example to obtain a non-strict inclusion in \eqref{level-sets-vs-gradients} for some \si\ functions. We obtain a weak formulation of Euler's homogeneous function theorem for some \si\ functions in the following proposition.

\begin{proposition}
Let $\f$ be a continuously differentiable \si\ function that can be written as $\varphi\circ p$ where $\varphi$ is a homeomorphism, $p$ is \phone\ and $\varphi$ and $\varphi^{-1}$ are continuously differentiable.
 Assume that the function $\foncfast{\R_{+}}{t}{ t \varphi^{\prime}(t) }{\R}$ is injective. Then for $x\in \R^{\n}$,
\begin{align}
\mathcal{L}_{f,x} = \level_{z\mapsto \nabla f(z) \cdot z, x}
\enspace.
\end{align}
\end{proposition}

\begin{proof}
It follows from Proposition~\ref{general-euler-si} that for all $x \in \R^{\n}$, $\nabla f(x) \cdot x = \varphi^{\prime}(p(x)) p(x)$.
Thanks to the bijectivity of $\varphi$ along with the injectivity of $t \to t\varphi^{\prime}(t)$, we have: 
\begin{align*}
 \varphi^{\prime}(p(y)) p(y) = \varphi^{\prime}(p(x)) p(x) \iff  p(x) = p(y) \iff \f(x) = \f(y).
 \end{align*}
In other words,
$\mathcal{L}_{f,x} = \acco{y \in \R^{\n}, \nabla \f(y) \cdot y = \nabla \f(x) \cdot x }.
$
\qed \end{proof}

\subsection{Compact Neighborhoods of Level Sets with Non-Vanishing Gradient}
\label{sphere-like}

We prove in this section that any continuously differentiable \si\ function $\f$ with a unique global argmin has level sets, for example $\level_{\f, z_{0}}$, such that for some compact neighborhood of the level set, $\mathcal{N}\supset\level_{\f, z_{0}}$, the gradient does not vanish and $\nabla \f (z) \cdot z > 0$ for all $z\in \mathcal{N}$.


For a continuously differentiable \ph\ function $p$ such that $p(x) > 0$ for all $x \neq 0$, \ie, such that $0$ is the unique global argmin of $p$, this result is a consequence of Euler's homogeneous function theorem which implies that  
\begin{align}
 \nabla p(x) \cdot x > 0 \text{ for all } x \neq 0
 \enspace.
 \label{positiveNabla}
\end{align}
In particular, \eqref{positiveNabla} is true on any compact neighborhood of any level set of $p$, if that compact does not contain $0$. 

We now remark that the property that $\nabla f \neq 0$ for all $x \neq 0$ is not necessarily true if $\f$ is a continuously differentiable \si\ function with a unique global argmin.
Namely, $f$ can have level sets that contain only saddle points.
\begin{lemma}\label{lemma:saddle}
%
%
Let $p (z) = \| z\|^{2}$ and $\varphi (t) =
\dsp\int_{0}^{t} \sin^{2}(u) \diff u$ for $t \geq 0$.
Then $\f = \varphi \circ p$ is a continuously differentiable \si\ function with a unique global argmin and an infinite number of $z$ belonging to different level sets of $\f$, such that $\nabla \f(z) = 0$.
\end{lemma}
\begin{proof}
The function $\varphi$ is strictly increasing since $\sin^{2}$ is non-negative and has zeros on isolated points. Also, for all $t \geq 0$, $\varphi(t) = \dsp\frac t2 -  \dsp\frac{\sin(2t)}{4}$, where we use that $\cos(2t) = 1 - 2 \sin^{2}(t)$.

For any natural integer $n$, $n \pi$ is a stationary point of inflection of $\varphi$: $\varphi^{\prime}(n\pi) = 0$ and $\varphi^{\prime\prime}(t)= \sin(2t)$ has opposite signs in the neighborhood of $n\pi$. 
 For all $z$ with $\|z\|^{2} \in \pi \Z_{+}$, $\nabla f(z) = \varphi^{\prime}(g(z)) \nabla g(z)= 2 \varphi^{\prime}(\|z\|^{2}) z = 0$.

Hence there exists an infinite number of level sets $\level_{\f, z}$ for which $\nabla \f (z) = 0$. \qed
\end{proof}

 Yet, a consequence of Theorem \ref{equivalences} and Lemma~\ref{lemma-level-sets-vs-gradients} is the existence of a level set of $\f$ such that $\nabla \f(z) \cdot z > 0$ for all $z$ in that level set as shown in the next proposition\nnew{.}

\begin{proposition}
Let $\f$ be a continuously differentiable \si\ function with $0$ as unique global argmin. There exists $z_{0} \in \overline{\B}_1$ with $\level_{\f, z_{0}} \subset \overline{\B}_1$, such that for all $z \in \level_{\f, z_{0}}$,
$\nabla \f(z) \cdot z > 0 $.
\label{no-vanish-gradients}
\end{proposition}
\begin{proof} Since $\f$ is a continuous \si\ function, we have all the equivalences in Theorem \ref{equivalences}.

Inside the proof of Theorem \ref{equivalences}, we have shown that there exists $s\in \S_{1}$ such that $\sublevel_{\f,s} \subset \overline{\B}_1$, with $ \f(s) = \dsp\min_{z\in\S_{1}} \f(z)$. Since $\f_{s}$ is strictly increasing and differentiable, there exists $t \in (0, 1]$ such that $\f_{s}^{\prime}(t) > 0$. Let us denote $z_{0} = t s$.
We have that $\mathcal{L}_{\f,z_{0}} \subset \sublevel_{\f, s} \subset \overline{\B}_1$. And with the chain rule, $\dsp 0 <  \f_{s}^{\prime}(t) = \nabla \f(z_{0}) \cdot \frac{z_{0}}{t}$. Therefore along with Lemma~\ref{lemma-level-sets-vs-gradients}, it follows that for all $z \in \mathcal{L}_{\f,z_{0}}$, $\nabla \f(z) \cdot z = \nabla \f(z_{0}) \cdot z_{0} > 0$.
\qed \end{proof}

From the uniform continuity of $z \mapsto \nabla f(z) \cdot z$ on a compact we deduce the announced result.
\begin{proposition}
Let $\f$ be a continuously differentiable \si\ function with $0$ as unique global argmin.
There exists $\delta > 0$, $z_{0} \in \overline{\B}_1$ with $\level_{\f, z_{0}} \subset \overline{\B}_1$ such that for all $z \in \mathcal{L}_{\f,z_{0}}  + \overline{\B(0, \delta)}$, $\nabla f(z) \cdot z > 0$.
\label{nozero}
\end{proposition}
\begin{proof}
Since $\nabla f(z) \cdot z > 0$ for all $z$ in the compact $\mathcal{L}_{\f,z_{0}}$, then $z \mapsto \nabla f(z) \cdot z$ has a positive minimum (that is reached) denoted by $\epsilon
 = \min_{z\in \mathcal{L}_{\f,z_{0}}} \nabla f (z) \cdot z  > 0$.
The continuous function $z \mapsto \nabla f(z) \cdot z$ is uniformly continuous on the compact $ \mathcal{L}_{\f,z_{0}} + \overline{\B(0, 1)}$, therefore there exists a positive number $\delta < 1$ such that if $y, z \in  \mathcal{L}_{\f,z_{0}} + \overline{\B(0, 1)}$ with $\norm{y-z} \leq  \delta$ then $\dsp\abs{ \nabla f(z) \cdot z - \nabla f(y) \cdot y} < \frac{\epsilon}{2}$. Then for all $z\in  \mathcal{L}_{\f,z_{0}} + \overline{\B(0, \delta)}$, there exists $y\in  \mathcal{L}_{\f,z_{0}}$ such that $\dsp\abs{\nabla f(z) \cdot z - \nabla f(y) \cdot y} < \frac{\epsilon}{2}$. Then $\nabla f(z) \cdot z > \nabla f(y) \cdot y - \frac{\epsilon}{2} \geq \frac{\epsilon}{2} > 0$.
Hence $z \mapsto \nabla f(z) \cdot z$ is positive on the compact set $ \mathcal{L}_{\f,z_{0}} + \overline{\B(0, \delta)}$.
\qed \end{proof}

\def\hyph{-\penalty0\hskip0pt\relax}
\newcommand{\SI}{scaling\hyph invariant}

\section{Summary and Conclusion}
This paper reveals that \emph{continuous} scaling-invariant functions are strictly monotonic transformations of \emph{continuous} positively homogeneous functions.
Moreover, we
present necessary and sufficient conditions for
any
\SI\ function to be a strictly monotonic transformation of a positively homogeneous function.
The derivation is solely based on
analyzing restrictions
to the half-lines starting from zero
that need to be strictly monotonic on a nontrivial interval
(or entirely flat).
We also highlight counter-intuitive examples of
scaling-invariant functions that are not monotonic on any nontrivial interval.

We then
present different properties of the level sets of a scaling-invariant function. In particular, Proposition~\ref{nozero} shows that 
continuously differentiable \SI\ functions with a unique
argmin
have a compact level set in a compact neighborhood with non-vanishing gradient.
The level set intersects
any
half-line with origin zero at a single
point---forming a ``star-shaped'' mani\-fold.
%

Scaling-invariant functions play a central role in the analysis of the convergence of some comparison-based stochastic optimization algorithms \cite{auger2016linear}. On this function class, for some translation and scale invariant comparison-based algorithms,
linear convergence can be deduced
when
analyzing the stability of a normalized process\footnote{%
In the case of step-size adaptive algorithms where the state of the algorithm equals a current
solution and a step-size, the normalized process equals to the
solution minus the reference point $x^\star$ of the scaling-invariant function, divided by the step-size.
Stability of the normalized process
is key to imply
linear convergence of the adaptive algorithm.}.
When linear convergence occurs,
the step-size and
the distance of the current
solution to
the optimum
decrease geometrically fast to zero at the same (linear) rate.

A stability analysis leading to linear convergence can be carried out for composites of strictly increasing functions with continuously differentiable scaling-invariant functions.
To obtain basic stability properties deduced from a connection
to
a deterministic control model \cite{chotard2019verifiable},
one can
exploit that
these
functions have Lebesgue negligible level sets
as a consequence of
Proposition~\ref{zeroLevel}.
In addition, the stability study relies on proving that
when the normalized process diverges,
the step-size multiplicative factor converges in distribution
to the
factor on nontrivial linear functions.
The proof exploits level set properties shown in
Proposition~\ref{nozero} and Corollary~\ref{unique-intersection-level-sets}.


\section*{Acknowledgements}

Part of this research has been conducted in the context of a research collaboration between Storengy and Inria. We particularly thank F.~Huguet and A.~Lange from Storengy for their strong support. We would like to thank the anonymous referee for their very careful reviewing of the paper and their suggestions that lead to an improved version.

\bibliographystyle{spmpsci}
\bibliography{scaling-invariance}

\begin{thebibliography}{10}
\providecommand{\url}[1]{{#1}}
\providecommand{\urlprefix}{URL }
\expandafter\ifx\csname urlstyle\endcsname\relax
  \providecommand{\doi}[1]{DOI~\discretionary{}{}{}#1}\else
  \providecommand{\doi}{DOI~\discretionary{}{}{}\begingroup
  \urlstyle{rm}\Url}\fi

\bibitem{aczel1966lectures}
Acz{\'e}l, J.: Lectures on {F}unctional {E}quations and their {A}pplications.
\newblock Academic Press, New York (1966)

\bibitem{auger2013linear}
Auger, A., Hansen, N.: Linear convergence on positively homogeneous functions
  of a comparison based step-size adaptive randomized search: the (1+1)-{ES}
  with generalized one-fifth success rule.
\newblock arXiv e-prints 1310.8397 [cs.NA]  (2013)

\bibitem{auger2016linear}
Auger, A., Hansen, N.: Linear convergence of comparison-based step-size
  adaptive randomized search via stability of {M}arkov chains.
\newblock SIAM Journal on Optimization \textbf{26}(3), 1589--1624 (2016)

\bibitem{buskes1997topological}
Buskes, G., van Rooij, A.: Topological {S}paces.
\newblock Springer-Verlag New York (1997)

\bibitem{chotard2019verifiable}
Chotard, A., Auger, A.: Verifiable conditions for the irreducibility and
  aperiodicity of {M}arkov chains by analyzing underlying deterministic models.
\newblock Bernoulli \textbf{25}(1), 112--147 (2019)

\bibitem{denjoy1915fonctions}
Denjoy, A.: Sur les fonctions d{\'e}riv{\'e}es sommables.
\newblock Bulletin de la Soci{\'e}t{\'e} Math{\'e}matique de France
  \textbf{43}, 161--248 (1915)

\bibitem{dutta2004monotonic}
Dutta, J., Martinez-Legaz, J., Rubinov, A.: Monotonic analysis over cones: I.
\newblock Optimization \textbf{53}(2), 129--146 (2004)

\bibitem{fournier2011lower}
Fournier, H., Teytaud, O.: Lower bounds for comparison based evolution
  strategies using {VC}-dimension and sign patterns.
\newblock Algorithmica \textbf{59}(3), 387--408 (2011)

\bibitem{gorokhovik2016positively}
Gorokhovik, V.V., Trafimovich, M.: Positively homogeneous functions revisited.
\newblock Journal of Optimization Theory and Applications \textbf{171}(2),
  481--503 (2016)

\bibitem{gorokhovik2018saddle}
Gorokhovik, V.V., Trafimovich, M.: Saddle representations of positively
  homogeneous functions by linear functions.
\newblock Optimization Letters \textbf{12}(8), 1971--1980 (2018)

\bibitem{hardy1916weierstrass}
Hardy, G.H.: Weierstrass's non-differentiable function.
\newblock Trans. Amer. Math. Soc. \textbf{17}(3), 301--325 (1916)

\bibitem{kuczma2009introduction}
Kuczma, M.: An {I}ntroduction to the {T}heory of {F}unctional {E}quations and
  {I}nequalities: {C}auchy's {E}quation and {J}ensen's {I}nequality.
\newblock Birkh\"auser (2009)

\bibitem{lasserre2002mathematical}
Lasserre, J.B., Hiriart-Urruty, J.B.: Mathematical properties of optimization
  problems defined by positively homogeneous functions.
\newblock Journal of optimization theory and applications \textbf{112}(1),
  31--52 (2002)

\bibitem{morinaga2019generalized}
Morinaga, D., Akimoto, Y.: Generalized drift analysis in continuous domain:
  linear convergence of (1+1)-{ES} on strongly convex functions with
  {L}ipschitz continuous gradients.
\newblock In: Proceedings of the 15th ACM/SIGEVO Conference on Foundations of
  Genetic Algorithms, pp. 13--24. Association for Computing Machinery (2019)

\bibitem{muresan2009concrete}
Muresan, M.: A {C}oncrete {A}pproach to {C}lassical {A}nalysis.
\newblock CMS Books in Mathematics. Springer-Verlag New York (2009)

\bibitem{nelder1965simplex}
Nelder, J.A., Mead, R.: A simplex method for function minimization.
\newblock The Computer Journal \textbf{7}(4), 308--313 (1965)

\bibitem{rubinov2003strictly}
Rubinov, A., Gasimov, R.: Strictly increasing positively homogeneous functions
  with application to exact penalization.
\newblock Optimization \textbf{52}(1), 1--28 (2003)

\bibitem{rubinov1998duality}
Rubinov, A., Glover, B.: Duality for increasing positively homogeneous
  functions and normal sets.
\newblock RAIRO - Operations Research \textbf{32}(2), 105--123 (1998)

\end{thebibliography}

\appendix

\section{Bijection Theorem \label{sec:appendixA}}

This standard theorem is reminded for the sake of completeness.

\begin{theorem}[Bijection theorem, {\cite[Theorem~2.20]{muresan2009concrete}}]
Let $I \subset \R$ be a nontrivial interval, $J \subset \R$ and $\varphi: I  \to J$ be a continuous bijection (and therefore strictly monotonic). Then $J$ is an interval and $\varphi$ is a homeomorphism, \ie, $\varphi^{-1}: J  \to I$ is also a continuous bijection, and if $\varphi$ is strictly increasing (respectively strictly decreasing), then $\varphi^{-1}$ is strictly increasing (respectively strictly decreasing). 
\label{homeomorphismTheorem}
\end{theorem}

\end{document}